\documentclass
[
    a4paper,
    DIV=14,
    abstract=true,
]
{scrartcl}

\usepackage
{
    amsmath,
    amssymb,
    amsthm,
    authblk,
    dsfont, % for 1 vector or indicator function
    enumitem,
    graphicx,
    kbordermatrix,
    mathtools,
    todonotes,
}

\usepackage[utf8]{inputenc}

\usepackage[pdffitwindow=false,
            plainpages=false,
            pdfpagelabels=true,
            pdfpagemode=UseOutlines,
            pdfpagelayout=SinglePage,
            bookmarks=false,
            colorlinks=true,
            hyperfootnotes=false,
            hypertexnames=false,
            linkcolor=blue,
            urlcolor=blue!30!black,
            citecolor=green!50!black]{hyperref}

\usepackage[bf,normal]{caption}
\usepackage[normalem]{ulem}

\graphicspath{{pics/}}

\newcommand{\subfiguretitle}[1]{{\scriptsize{#1}} \\}
\newcommand{\R}{\mathbb{R}}                                      % real numbers
\newcommand{\pd}[2]{\frac{\partial#1}{\partial#2}}               % partial derivatives
\newcommand{\innerprod}[2]{\left\langle #1,\, #2 \right\rangle}  % scalar product
\newcommand{\ts}{\hspace*{0.1em}}                                % thin space
\providecommand{\norm}[1]{\left\lVert #1 \right\rVert}           % norm
\providecommand{\vdot}{\boldsymbol\cdot}                         % dot product

\newcommand*\circled[1]{\tikz[baseline=-3pt]{\node[shape=circle,draw,inner sep=1.2pt,color=red] (char) {\tiny \textcolor{red}{#1}};}}

\newcommand\xqed[1]{\leavevmode\unskip\penalty9999 \hbox{}\nobreak\hfill \quad\hbox{#1}}
\newcommand{\exampleSymbol}{\xqed{$\triangle$}}

\DeclareMathOperator{\tr}{tr}
\DeclareMathOperator{\mspan}{span}
\DeclareMathOperator*{\argmin}{arg\,min}
\let\vec\relax
\DeclareMathOperator{\vec}{vec}

\newtheorem{theorem}{Theorem}[section]

\newtheorem{lemma}[theorem]{Lemma}

\newtheorem{definition}[theorem]{Definition}
\theoremstyle{definition}
\newtheorem{example}[theorem]{Example}
\newtheorem{remark}[theorem]{Remark}

\makeatletter
\renewcommand*\env@matrix[1][*\c@MaxMatrixCols c]{%
  \hskip -\arraycolsep
  \let\@ifnextchar\new@ifnextchar
  \array{#1}}
\makeatother

\mathtoolsset{centercolon} % for := and =:

\allowdisplaybreaks

\DeclareCaptionLabelFormat{period}{#1~#2}
\makeatletter
\def\blfootnote{\gdef\@thefnmark{}\@footnotetext}
\makeatother

\begin{document}

\title{Data-driven system identification using \\ quadratic embeddings of nonlinear dynamics}
\author[1]{Stefan Klus}
\author[2]{Joel-Pascal Ntwali N'konzi}
\affil[1]{School of Mathematical \& Computer Sciences, Heriot--Watt University, Edinburgh, UK}
\affil[2]{Maxwell Institute for Mathematical Sciences, University of Edinburgh and Heriot--Watt University, Edinburgh, UK}

\date{}

\maketitle

\begin{abstract}
We propose a novel data-driven method called QENDy (\emph{Quadratic Embedding of Nonlinear Dynamics}) that not only allows us to learn quadratic representations of highly nonlinear dynamical systems, but also to identify the governing equations. The approach is based on an embedding of the system into a higher-dimensional feature space in which the dynamics become quadratic. Just like SINDy (\emph{Sparse Identification of Nonlinear Dynamics}), our method requires trajectory data, time derivatives for the training data points, which can also be estimated using finite difference approximations, and a set of preselected basis functions, called \emph{dictionary}. We illustrate the efficacy and accuracy of QENDy with the aid of various benchmark problems and compare its performance with SINDy and a deep learning method for identifying quadratic embeddings. Furthermore, we analyze the convergence of QENDy and SINDy in the infinite data limit, highlight their similarities and main differences, and compare the quadratic embedding with linearization techniques based on the Koopman operator.
\end{abstract}

\section{Introduction}

Over the last few years, many data-driven approaches for the analysis of complex dynamical systems have been proposed that are based on the Koopman operator or its infinitesimal generator. The Koopman operator provides a linear representation of the typically highly nonlinear dynamics, describing the system in terms of \emph{observables} and their evolution in time \cite{Ko31, LaMa94, Mezic05}. Its eigenvalues and eigenfunctions, often estimated from trajectory data, contain important information about global properties such as inherent time scales and associated spatiotemporal patterns \cite{BMM12, WKR15, KKS16}. Applications include the detection of metastable sets, system identification, stability analysis, forecasting, and control \cite{SS13, MauMez16, MauGon16, KorMez16, PeiKlul19, KNPNCS20, BSH21}, see \cite{KD24} for a detailed overview and additional references. However, the linearity of the Koopman operator comes at a cost---namely, the representation is \emph{infinite-dimensional}. The Koopman operator framework is also closely related to Carleman linearization techniques \cite{Car32, KS91, DY24}. If we allow not only linear but also quadratic representations, it is always possible to obtain a \emph{finite-dimensional} model as shown in \cite{Ker81, Gu11}, provided that the system comprises only elementary functions such as polynomial, trigonometric, exponential, logarithmic, and rational functions (including compositions thereof). The reformulation process---sometimes called \emph{quadratization} or in general \emph{polynomialization}---is based on successively defining new variables and applying the chain rule. Similar ideas have also been used in \cite{App02, SV87, PP05, LZZZ15, Car15, KW19, CP24, BIPK24}. Bounds on the number of required additional variables are known, see, e.g., \cite{Gu11}, but the quadratic embedding itself is in general not uniquely defined. A quadratic embedding of the dynamics can be regarded as a compromise between a highly nonlinear but low-dimensional model and a linear but infinite-dimensional model.

Instead of using the Koopman operator or generator for system identification and forecasting, we develop a method that is similar in spirit, but relies on a quadratic embedding instead of a Koopman or Carleman linearization. The main goal is again to reformulate nonlinear dynamical systems using a model structure that is as simple as possible but still able to capture complicated dynamical behavior and can be easily estimated from simulation or measurement data. These quadratic systems can then be combined with model reduction, optimization, model predictive control, or uncertainty quantification techniques tailored to such problems. A deep learning method for identifying quadratic embeddings was proposed in \cite{GB24}. While such an approach is extremely flexible and versatile, the learned dynamical systems typically lack interpretability and often do not generalize well to unseen data. Another drawback is that the method requires solving an optimization problem involving a linear combination of three potentially contradictory loss functions so that we obtain only an \emph{optimal compromise}, i.e., one point of the Pareto set (see Section \ref{sec:dl-quadembed} for more details).

Our method, which we call \emph{Quadratic Embedding of Nonlinear Dynamics} (QENDy), on the other hand, is closely related to the \emph{Sparse Identification of Nonlinear Dynamics} (SINDy) method~\cite{BPK16}, which directly estimates the governing equations and can be regarded as a special case of a Koopman generator estimation as shown in \cite{KNPNCS20}. This allows us to not only identify the quadratic embedding, but also to extract the governing equations, which can then be written in terms of the functions contained in the dictionary. Furthermore, we show that an optimal solution of the resulting optimization problem can be obtained by solving systems of linear equations. In this sense, our method is closely related to the operator inference framework for learning reduced-order models from data by postulating a polynomial dynamical system in the reduced space \cite{PW16}. An open question---regarding not only QENDy and SINDy, but also methods for approximating the Koopman operator such as \emph{Extended Dynamic Mode Decomposition} (EDMD) and its extensions~\cite{NoNu13, WKR15, KKS16}---is the choice of suitable basis functions. Given a sufficiently large dictionary, SINDy uses sparse approximation techniques to identify parsimonious models. That is, the goal is to represent the vector field by a linear combination of only a few basis functions. QENDy, on the other hand, uses a dictionary to define a typically nonlinear mapping from the original state space to a higher-dimensional embedding space and then learns a quadratic model in this embedded space. Relationships between these different modeling approaches will be discussed in more detail below. The main contributions of this work are as follows:
\begin{enumerate}[leftmargin=3.5ex, itemsep=0ex, topsep=0.5ex, label=\roman*)]
\item We first derive a method called QENDy for learning quadratic embeddings of nonlinear ordinary differential equations, given a set of training data points and a dictionary containing basis functions.
\item We analyze the convergence of QENDy and SINDy in the infinite data limit and show that both methods can be interpreted in terms of best approximations.
\item We present numerical results for benchmark problems and compare the efficacy and accuracy of QENDy, SINDy, and an autoencoder-based method for learning quadratic embeddings.
\end{enumerate}
We will start with an introduction to the quadratic embedding framework and several guiding examples in Section~\ref{sec:Quadratic reformulation} and then derive an optimization problem that allows us to identify the optimal quadratic embedding for a given feature space in Section~\ref{sec:QENDy}. Numerical results will be presented in Section~\ref{sec:Numerical results} and concluding remarks and open problems in Section~\ref{sec:Conclusion}.

\section{Quadratic reformulation of nonlinear dynamical systems}
\label{sec:Quadratic reformulation}

Given an autonomous ordinary differential equation
\begin{equation} \label{eq:ODE}
    \dot{x} = F(x),
\end{equation}
with $ F \colon \R^n \to \R^n $, we embed the state $ x $ into an $ N $-dimensional feature space via
\begin{equation*}
    z := \phi(x) =
    \begin{bmatrix}
        \phi_1(x) \\
        \vdots \\
        \phi_N(x)
    \end{bmatrix},
\end{equation*}
where $ \phi \colon \R^n \to \R^N $ and $ N > n $. We call $ \{ \phi_i  \}_{i=1}^N $ the \emph{set of basis functions} or \emph{dictionary}. Using the chain rule, we obtain
\begin{equation*}
    \dot{z} =
    \underbrace{
    \begin{bmatrix}
        \nabla \phi_1^\top(x) \\
        \vdots \\
        \nabla \phi_N^\top(x)
    \end{bmatrix}}_{=: J(x)}
    \dot{x} = J(x) \ts F(x).
\end{equation*}
The idea is to lift the original dynamics to a higher-dimensional space in such a way that we obtain a quadratic ordinary differential equation of the form
\begin{equation} \label{eq:quadratic ODE}
    \dot{z} = A (z \otimes z) + B \ts z + C,
\end{equation}
with $ A \in \R^{N \times N^2} $, $ B \in \R^{N \times N} $, and $ C \in R^N $, where $ \otimes $ denotes the Kronecker product. Any nonlinear ordinary differential equation---provided it can be written in terms of compositions of elementary functions---can be reformulated as a quadratic system in a higher-dimensional feature space, see, e.g., \cite{Ker81, Gu11}.

\begin{remark}
Before we continue with an illustrative example, a few remarks are in order:
\begin{enumerate}[leftmargin=3.5ex, itemsep=0ex, topsep=0.5ex, label=\roman*)]
\item Let $ \mathcal{L} $ denote the infinitesimal generator of the Koopman operator associated with \eqref{eq:ODE}, which for an observable $ f $ is defined by
\begin{equation*}
    \mathcal{L} f = F \vdot \nabla f,
\end{equation*}
we can also write the quadratic system as $ \dot{z} = \mathcal{L} \phi $, where the generator is applied component-wise.
\item Since $ z \otimes z $ contains $ z_i \ts z_j $ and also $ z_j \ts z_i $, some columns of the matrix $ A $ are clearly redundant. Reshaping each row of $ A $ as an $ N \times N $ matrix, we could simply store only the upper-triangular part or assume w.l.o.g.\ that the matrix is symmetric, reducing the overall number of degrees of freedom from $ N^3 $ to $ \frac{N^2(N+1)}{2} $. However, for the sake of clarity, we will work with the full matrix~$ A $.
\item Instead of using a matrix $ A \in \R^{N \times N^2} $, we could also define a tensor $ A \in R^{N \times N \times N} $. This would allow us to use low-rank tensor approximation techniques (e.g., matrix product states \cite{Ban06} or, equivalently, the tensor-train format \cite{Ose11}) and to potentially mitigate the curse of dimensionality.
\end{enumerate}
\end{remark}

\begin{example}
Given the differential equation $ \dot{x} = \frac{1}{1 + e^{x}} $, we first define $ \dot{z}_1 = z_2 $ and $ z_2 = \frac{1}{1 + e^{z_1}} $, which implies
\begin{align*}
    \dot{z}_2 &= -\frac{1}{(1 + e^{z_1})^2} \ts e^{z_1} \dot{z}_1 = -z_2^3 \ts e^{z_1}. \\
\intertext{Adding $ z_3 = e^{z_1} $, we then obtain}
    \dot{z}_1 &= z_2, \\
    \dot{z}_2 &= -z_2^3 \ts z_3, \\
    \dot{z}_3 &= e^{z_1} \ts \dot{z}_1 = z_2 \ts z_3.
\intertext{Finally, defining $ z_4 = z_2^2 \ts z_3 $ yields}
    \dot{z}_1 &= z_2, \\
    \dot{z}_2 &= -z_2 \ts z_4, \\
    \dot{z}_3 &= z_2 \ts z_3, \\
    \dot{z}_4 &= 2 \ts z_2 \ts \dot{z}_2 \ts z_3 + z_2^2 \ts \dot{z}_3 = -2 \ts z_4^2 + z_2 \ts z_4,
\end{align*}
which is a quadratic model of the form \eqref{eq:quadratic ODE}. \exampleSymbol
\end{example}

In what follows, we assume that the original state space variables $ x_1, \dots, x_n $ can be written as linear combinations of the feature space variables $ z_1, \dots, z_N $, which implies that there exists a matrix $ G \in \R^{n \times N} $ such that
\begin{equation*}
    x = G \ts z.
\end{equation*}
This is similar to the assumption that the so-called full-state observable $ g(x) = x $, which for Koopman operator-based methods such as EDMD and gEDMD is required for the decomposition of the dynamics into eigenvalues, eigenfunctions, and modes, is contained in the span of the basis functions, see \cite{WKR15, KNPNCS20} for details. The easiest way to accomplish this is to add the observables $ \{ x_i \}_{i=1}^n $ to the dictionary. The above assumption then implies that
\begin{equation} \label{eq:projection of quadratic ODE}
    \dot{x} = G \ts \dot{z} = G \ts A (z \otimes z) + G \ts B \ts z + G \ts C.
\end{equation}
That is, once we have learned the quadratic embedding, we can also extract the governing equations of the original dynamical system. Note that after mapping the initial condition $ x(t_0) = x_0 $ to the feature space using $ z(t_0) = \phi(x_0) $, we do not need to evaluate the typically nonlinear functions contained in the dictionary anymore.

\begin{example} \label{ex:guiding examples}
Let us consider two simple guiding examples:
\begin{enumerate}[leftmargin=3.5ex, itemsep=0ex, topsep=0.5ex, label=\roman*)]
\item Given $ \dot{x} = -\frac{x}{1+x} $, see \cite{GB24}, we define
\begin{equation*}
    z =
    \begin{bmatrix}
        x \\[0.5ex]
        \frac{1}{1+x} \\[0.8ex]
        \frac{x}{(1+x)^2} \\[1.1ex]
    \end{bmatrix}
    \implies
    \dot{z} =
    \begin{bmatrix}
        1 \\[0.5ex]
        -\frac{1}{(1+x)^2} \\[0.8ex]
        \frac{1}{(1+x)^2} - \frac{2 \ts x}{(1 + x)^3} \\[1.1ex]
    \end{bmatrix}
    \left(-\tfrac{x}{1+x}\right)
    =
    \begin{bmatrix}
        -z_1 \ts z_2 \\
        z_2 \ts z_3 \\
        -z_2 \ts z_3 + 2 \ts z_3^2
    \end{bmatrix},
\end{equation*}
so that in this case $ G = \begin{bmatrix} 1 & 0 & 0 \end{bmatrix} $. We can rewrite the nonlinear system as a quadratic differential equation of the form \eqref{eq:quadratic ODE} in a three-dimensional space, with
\begin{equation*}
    A = \hspace*{-1.5ex}
    \kbordermatrix{
        & z_1^2 & z_1 \ts z_2 & z_1 \ts z_3 & z_2 \ts z_1 & z_2^2 & z_2 \ts z_3 & z_3 \ts z_1 & z_3 \ts z_2 & z_3^2 \\
        & 0 & -1 & 0 & 0 & 0 &  0 & 0 & 0 & 0 \\
        & 0 &  0 & 0 & 0 & 0 &  1 & 0 & 0 & 0 \\
        & 0 &  0 & 0 & 0 & 0 & -1 & 0 & 0 & 2
    },~~
    B = \hspace*{-1.5ex}
    \kbordermatrix{
        & z_1 & z_2 & z_3 \\
        & 0 & 0 & 0 \\
        & 0 & 0 & 0 \\
        & 0 & 0 & 0
    },~
    C = \hspace*{-1.5ex}
    \kbordermatrix{
        & \mathds{1} \\
        & 0  \\
        & 0  \\
        & 0
    }.
\end{equation*}
As already mentioned above, this representation is in general not unique, even if we remove the redundant columns from the matrix $ A $. We could, for instance, also write $ \dot{z}_1 = -z_1 \ts z_3 - z_3 $ or use any convex combination of the two corresponding matrix representations.
\item For the damped nonlinear pendulum, see also \cite{GB24}, given by
\begin{equation*}
    \begin{bmatrix}
        \dot{x}_1 \\
        \dot{x}_2 \\
    \end{bmatrix}
    =
    \begin{bmatrix}
        x_2 \\
        -\sin(x_1) - c \ts x_2
    \end{bmatrix},
\end{equation*}
where $ c $ is the damping constant, we define
\begin{equation*}
    z =
    \begin{bmatrix}
        x_1 \\
        x_2 \\
        \sin(x_1) \\
        \cos(x_1)
    \end{bmatrix}
    \implies
    \dot{z} =
    \begin{bmatrix}
        1 & 0 \\
        0 & 1 \\
        \cos(x_1) & 0 \\
        -\sin(x_1) & 0
    \end{bmatrix}
    \begin{bmatrix}
        x_2 \\
        -\sin(x_1) - c \ts x_2
    \end{bmatrix}
    =
    \begin{bmatrix}
        z_2 \\
        -z_3 - c \ts z_2 \\
        z_2 \ts z_4 \\
        -z_2 \ts z_3
    \end{bmatrix}.
\end{equation*}
The projection onto the original state space is then defined by the matrix
\begin{equation*}
    G =
    \begin{bmatrix}
        1 & 0 & 0 & 0 \\
        0 & 1 & 0 & 0
    \end{bmatrix}.
\end{equation*}
This shows that we can represent the nonlinear pendulum by a quadratic system in a four-dimensional space. \exampleSymbol
\end{enumerate}
\end{example}

\begin{remark}
The quadratic embedding approach is similar to the Koopman linearization, which seeks to find a linear but in general infinite-dimensional representation of the system:
\begin{enumerate}[leftmargin=3.5ex, itemsep=0ex, topsep=0.5ex, label=\roman*)]
\item First, consider
\begin{equation*}
    \begin{bmatrix}
        \dot{x}_1 \\
        \dot{x}_2 \\
    \end{bmatrix}
    =
    \begin{bmatrix}
        x_1 - x_2^4 \,\, \\
        2 \ts x_2
    \end{bmatrix}.
\end{equation*}
Defining $ z_1 = x_1 $, $ z_2 = x_2 $, and $ z_3 = x_2^4 $, i.e., $ \phi(x) = \big[x_1, x_2, x_2^4\big]^\top $, we can rewrite the nonlinear system as a linear system
\begin{equation*}
    \begin{bmatrix}
        \dot{z}_1 \\
        \dot{z}_2 \\
        \dot{z}_3
    \end{bmatrix}
    =
    \begin{bmatrix}
        1 & 0 & -1 \\
        0 & 2 & 0 \\
        0 & 0 & 8
    \end{bmatrix}
    \begin{bmatrix}
        z_1 \\
        z_2 \\
        z_3
    \end{bmatrix}
\end{equation*}
in three dimensions. The three observables $ \phi_i $, $ i = 1, \dots, 3 $, span in this case an invariant subspace of the Koopman generator associated with the original nonlinear system. Given a linear dynamical system $ \dot{z} = B \ts z $, we can construct eigenfunctions of the Koopman generator by computing left eigenvectors of the matrix $ B $. Assume that $ B^\top v = \lambda \ts v $, then $ \varphi(z) = z \vdot v $ is an eigenfunction since
\begin{equation*}
    \mathcal{L} \varphi(z) = (B \ts z) \vdot v = z \vdot (B^\top v) = z \vdot (\lambda \ts v) = \lambda \ts \varphi(z).
\end{equation*}
Products of these eigenfunctions are again eigenfunctions. We can use this approach to compute eigenfunctions of the original system. For the above system, for instance, we obtain the eigenfunction $ \varphi(x) = 7 \ts x_1 + x_2^4 $. In general, however, the projection of the Koopman generator onto a finite-dimensional subspace spanned by a set of basis functions leads to approximation errors, see \cite{KNPNCS20} for details.
\item For the modified system
\begin{equation*}
    \begin{bmatrix}
        \dot{x}_1 \\
        \dot{x}_2 \\
    \end{bmatrix}
    =
    \begin{bmatrix}
        x_1 - x_2^4 \,\, \\
        x_1 + 2 \ts x_2
    \end{bmatrix},
\end{equation*}
on the other hand, this linearization approach does not work: Defining again $ z_1 = x_1 $, $ z_2 = x_2 $, and $ z_3 = x_2^4 $, we have $ \dot{z}_3 = 4 \ts z_1 \ts z_2^3 + 8 \ts z_3 $, which would require introducing a new auxiliary variable $ z_4 = z_1 \ts z_2^3 $ so that $ \dot{z}_4 = z_4 - z_2^7 + 3 \ts z_1^2 + 3 \ts z_1 \ts z_2 $, which in turn would again require defining new variables. In this case, we would need infinitely many terms. In order to find a quadratic model, we can simply define
\begin{equation*}
    z =
    \begin{bmatrix}
        x_1 \\
        x_2 \\[0.5ex]
        x_2^2 \,\,
    \end{bmatrix}
    \implies
    \dot{z} =
    \begin{bmatrix}
        z_1 - z_3^2 \\
        z_1 + 2 \ts z_2 \\
        2 \ts z_1 \ts z_2 + 4 \ts z_3
    \end{bmatrix},
\end{equation*}
i.e., the system is quadratic in the chosen three-dimensional feature space. \exampleSymbol
\end{enumerate}
\end{remark}

The example illustrates that if we add new variables and apply the chain rule to derive a linear model, this process does in general not terminate. If we allow quadratic systems, on the other hand, it is always possible to determine a finite-dimensional model \cite{Ker81, Gu11}.

\section{Learning quadratic embeddings from data}
\label{sec:QENDy}

Our goal is to learn the quadratic embedding \eqref{eq:quadratic ODE} from simulation or measurement data. Assume we have $ m $ data points $ x^{(k)} $, extracted for instance from one long trajectory, and the corresponding time derivatives $ \dot{x}^{(k)} $. Our training data is thus given by $ \{ (x^{(k)}, \dot{x}^{(k)}) \}_{k=1}^m $. We either need to be able to evaluate the right-hand side pointwise to compute $ \dot{x}^{(k)} = F\big(x^{(k)}\big) $ or estimate the derivatives using finite difference approximations. In addition to the training data, we have to select a set of basis functions $ \{ \phi_1, \dots, \phi_N \} $. The gradients of the basis functions, which will be required later, can be computed analytically or using automatic differentiation.

\subsection{Derivation of QENDy} \label{sec:derivation_QENDy}

Let $ z^{(k)} = \phi\big(x^{(k)}\big) $ and $ \dot{z}^{(k)} = J\big(x^{(k)}\big) \ts \dot{x}^{(k)} $. In order to find the optimal $ A $, $ B $, and $ C $ in \eqref{eq:quadratic ODE} for the given set of basis functions, we minimize the loss function
\begin{equation*}
    L(A, B, C) = \sum_{k=1}^m \norm{ \dot{z}^{(k)} - A \big(z^{(k)} \otimes z^{(k)}\big) - B \ts z^{(k)} - C }_2^2.
\end{equation*}
We define the data matrices
\begin{align*}
    Z_1 &=
    \begin{bmatrix}
        z^{(1)} & z^{(2)} & \dots & z^{(m)}
    \end{bmatrix}
    \in \R^{N \times m}, \\
    Z_2 &=
    \begin{bmatrix}
        z^{(1)} \otimes z^{(1)} & z^{(2)} \otimes z^{(2)} & \dots & z^{(m)} \otimes z^{(m)}
    \end{bmatrix}
    \in \R^{N^2 \times m}, \\
    \dot{Z} &=
    \begin{bmatrix}
        \dot{z}^{(1)} & \dot{z}^{(2)} & \dots & \dot{z}^{(m)}
    \end{bmatrix}
    \in \R^{N \times m},
\end{align*}
so that we can rewrite the loss function as
\begin{equation} \label{eq:loss function}
    L(A, B, C) = \norm{\dot{Z} - A \ts Z_2 - B \ts Z_1 - C \ts \mathds{1}_m^\top}_F^2,
\end{equation}
where $ \norm{\ts\cdot\ts}_F $ denotes the Frobenius norm and $ \mathds{1}_m \in \R^m $ the vector of ones, which is here simply used to create $ m $ copies of the vector~$ C $.

\begin{theorem} \label{thm:derivatives}
The derivatives of the loss function \eqref{eq:loss function} with respect to $ A $, $ B $, and $ C $ are given by
\begin{align*}
    \pd{L(A, B, C)}{A} &= 2 \ts A \ts Z_2 \ts Z_2^\top - 2 \ts \dot{Z} \ts Z_2^\top + 2 \ts B \ts Z_1 \ts Z_2^\top + 2 \ts C \ts \mathds{1}_m^\top Z_2^\top, \\
    \pd{L(A, B, C)}{B} &= 2 \ts B Z_1 \ts Z_1^\top - 2 \ts \dot{Z} \ts Z_1^\top + 2 \ts A \ts Z_2 \ts Z_1^\top + 2 \ts C \ts \mathds{1}_m^\top Z_1^\top, \\
    \pd{L(A, B, C)}{C} &= 2 \ts m \ts C - 2 \ts \dot{Z} \mathds{1}_m + 2 \ts A \ts Z_2 \ts \mathds{1}_m + 2 \ts B \ts Z_1 \ts \mathds{1}_m.
\end{align*}
\end{theorem}

The proof can be found in Appendix \ref{app:Proof}. In order to find an optimal solution, the derivatives in Theorem~\ref{thm:derivatives} must be zero. Transposing the equations, we obtain
\begin{alignat}{3}
    Z_2 \ts Z_2^\top A^\top &+\,& Z_2 \ts Z_1^\top B^\top &+\,& Z_2 \ts \mathds{1}_m \ts C^\top &= Z_2 \ts \dot{Z}^\top, \nonumber \\
    Z_1 \ts Z_2^\top A^\top &+& Z_1 \ts Z_1^\top B^\top &+& Z_1 \ts \mathds{1}_m \ts C^\top &= Z_1 \ts \dot{Z}^\top, \label{eq:minimum conditions} \\
    \mathds{1}_m^\top Z_2^\top A^\top &+& \mathds{1}_m^\top Z_1^\top B^\top &+& m \ts C^\top &= \mathds{1}_m^\top \dot{Z}^\top \nonumber.
\end{alignat}

\begin{definition}[Vectorization]
Let $ X \in \R^{m \times n} $ be an arbitrary matrix. We define $ \vec(X) \in \R^{m \ts n} $ to be the vectorization of $ X $, which reshapes a matrix as a vector by stacking its columns.
\end{definition}

\begin{lemma}
Given two matrices $ X \in \R^{m \times n} $ and $ Y \in \R^{n \times \ell} $, it holds that
\begin{equation*}
    \vec(X \ts Y) = (I_\ell \otimes X) \vec(Y) \in \R^{m \ts l},
\end{equation*}
see, e.g., \cite{PP12}.
\end{lemma}

As a result, we can rewrite \eqref{eq:minimum conditions} as a system of linear equations
\begin{equation*}
    \renewcommand*{\arraystretch}{1.3}
    \left[\,
    \begin{matrix}[c|c|c]
        I_N \otimes \big(Z_2 \ts Z_2^\top\big) & I_N \otimes \big(Z_2 \ts Z_1^\top\big) & I_N \otimes \big(Z_2 \ts \mathds{1}_m\big) \\ \hline
        I_N \otimes \big(Z_1 \ts Z_2^\top\big) & I_N \otimes \big(Z_1 \ts Z_1^\top\big) & I_N \otimes \big(Z_1 \ts \mathds{1}_m\big) \\ \hline
        I_N \otimes \big(\mathds{1}_m^\top \ts Z_2^\top\big) & I_N \otimes \big(\mathds{1}_m^\top \ts Z_1^\top\big) & m \ts I_N
    \end{matrix}\,\right]
    \begin{bmatrix}
        \vec\big(A^\top\big) \\
        \vec\big(B^\top\big) \\
        \vec\big(C^\top\big)
    \end{bmatrix}
    =
    \begin{bmatrix}
        \vec\big(Z_2 \ts \dot{Z}^\top\big) \\
        \vec\big(Z_1 \ts \dot{Z}^\top\big) \\
        \vec\big(\mathds{1}_m^\top \ts \dot{Z}^\top\big)
    \end{bmatrix}
\end{equation*}
of size $ N^3 + N^2 + N $. Note, however, that the system can be decoupled into $ N $ systems of size $ N^2 + N + 1 $. Let $ A_\ell $, $ B_\ell $, and $ \dot{Z}_\ell $ denote the $ \ell $th rows of the matrices $ A $, $ B $, and $ \dot{Z} $, respectively, and let $ C_\ell $ be the $ \ell $th entry of the vector $ C $. We can then solve
\begin{equation} \label{eq:decomposed system}
    \renewcommand*{\arraystretch}{1.3}
    \underbrace{
    \left[\,
    \begin{matrix}[c|c|c]
        Z_2 \ts Z_2^\top & Z_2 \ts Z_1^\top & Z_2 \ts \mathds{1}_m \\ \hline
        Z_1 \ts Z_2^\top & Z_1 \ts Z_1^\top & Z_1 \ts \mathds{1}_m \\ \hline
        \mathds{1}_m^\top \ts Z_2^\top & \mathds{1}_m^\top \ts Z_1^\top & m
    \end{matrix}\,\right]}_{R}
    \underbrace{
    \begin{bmatrix}
        A_\ell^\top \\
        B_\ell^\top \\
        C_\ell
    \end{bmatrix}}_{v_\ell}
    =
    \underbrace{
    \begin{bmatrix}
        Z_2 \ts \dot{Z}_\ell^\top \\
        Z_1 \ts \dot{Z}_\ell^\top \\
        \mathds{1}_m^\top \ts \dot{Z}_\ell^\top
    \end{bmatrix}}_{s_\ell}
\end{equation}
for $ \ell = 1, \dots, N $. Since the quadratic embedding is in general not unique---even if we eliminate the redundant columns of the matrix $ A $---as shown in Example \ref{ex:guiding examples}, the matrix $ R $ is not necessarily invertible so that we define $ v_\ell = R^+ s_\ell $, where $ R^+ $ denotes the pseudoinverse of the symmetric matrix $ R $. That is, $ v_\ell $ is the solution with the smallest 2-norm. Solving the $ N $ systems of equations allows us to compute $ A $, $ B $, and $ C $ row-wise. We call this method \emph{QENDy}.

\begin{remark} \label{rem:regression problem}
Defining
\begin{equation*}
    U =
    \begin{bmatrix}
        Z_2^\top & Z_1^\top & \mathds{1}_m
    \end{bmatrix}
    \in \R^{m \times (N^2 + N + 1)}
\end{equation*}
so that $ R = U^\top U $ and $ s_\ell = U^\top \dot{Z}_\ell^\top $, we can view \eqref{eq:decomposed system} as the normal equation corresponding to the regression problem $ U \ts v_\ell = \dot{Z}_\ell^\top $, whose solution is given by $ v_\ell = U^+ \dot{Z}_\ell^\top $. Assuming $ m > N^2 + N + 1 $, computing the pseudoinverse of $ U $ rather than $ R $ is more expensive, but in general also more accurate. The computational complexity is then $ \mathcal{O}\big(m(N^2 + N + 1)^2\big) $.
\end{remark}

\begin{example}
Let us consider the systems along with the dictionaries defined in Example~\ref{ex:guiding examples}:
\begin{enumerate}[leftmargin=3.5ex, itemsep=0ex, topsep=0.5ex, label=\roman*)]
\item We generate eleven training data points by simulating the rational system for $ t \in [0, 5] $ using the initial condition $ x(0) = 1 $. To avoid too many redundant solutions, we set $ C = 0 $. Applying QENDy then yields
\begin{equation*}
    \dot{z} =
    \begin{bmatrix}
        -0.7 \ts z_1 \ts z_2 - 0.3 \ts z_1 \ts z_3 + 0.1 \ts z_2^2 - 0.1 \ts z_2 - 0.2 \ts z_3 \\
        z_2 \ts z_3 \\
        -z_2 \ts z_3 + 2 \ts z_3^2
    \end{bmatrix}
    \implies
    \dot{x} = -\frac{x}{1+x}.
\end{equation*}
We obtain a valid quadratic embedding and can also identify the governing equations (after simplifying the right-hand side) using \eqref{eq:projection of quadratic ODE}. It would be possible to represent $ \dot{z}_1 $ using fewer terms, this could be accomplished by adding sparsity constraints to the loss function.
\item For the nonlinear damped pendulum, we choose the damping coefficient $ c = 0.1 $. Even when the number of training data points is small, QENDy returns
\begin{equation*}
    \dot{z} =
    \begin{bmatrix}
        z_2 \\
        -0.1 \ts z_2 - z_3 \\
        z_2 \ts z_4 \\
        -z_2 \ts z_3
    \end{bmatrix}
    \implies
    \dot{x} =
    \begin{bmatrix}
        x_2 \\
        -0.1 \ts x_2 - \sin(x_1)
    \end{bmatrix}.
\end{equation*}
The convergence will be analyzed in more detail in Example~\ref{ex:convergence}. \exampleSymbol
\end{enumerate}

\end{example}

In general, we do not know the feature space mapping that results in a perfect quadratic embedding of the dynamics, but we will nevertheless obtain the best approximation for a given set of basis functions and training data points as we will show next. One of the main advantages of QENDy, compared to deep learning methods, is that it generalizes well to unseen data and also allows us to explicitly write the governing equations in terms of the basis functions.

\subsection{Convergence of QENDy}

Let us now analyze the convergence of QENDy in the infinite data limit, assuming that the training data points $ x^{(k)} $ are sampled from a (potentially unknown) distribution $ \mu $. In order to compute the $ \ell $th rows of $ A $, $ B $, and $ C $, we have to solve the system of linear equations \eqref{eq:decomposed system}. First, we divide both sides by the number of data points $ m $. Let $ (i, j) \mapsto \overline{ij} := N \ts (j-1) + i $ be the canonical mapping from $ \{1, \dots, N\} \times \{1, \dots, N\} $ to $ \{1, \dots, N^2\} $, then
\begin{equation*}
    \frac{1}{m} Z_2 \ts Z_2^\top = \frac{1}{m} \sum_{k=1}^{m} \big(z^{(k)} \otimes z^{(k)}\big) \big(z^{(k)} \otimes z^{(k)}\big)^\top
\end{equation*}
and the entry in row $ \overline{i_1 i_2} $ and column $ \overline{j_1 j_2} $ is
\begin{align*}
    \left[\frac{1}{m} Z_2 \ts Z_2^\top\right]_{\overline{i_1 i_2}\; \overline{j_1 j_2}}
        &= \frac{1}{m} \sum_{k=1}^m z_{i_1}^{(k)} \ts z_{i_2}^{(k)} \ts z_{j_1}^{(k)} \ts z_{j_2}^{(k)}
        = \frac{1}{m} \sum_{k=1}^m \phi_{i_1}\big(x^{(k)}\big) \phi_{i_2}\big(x^{(k)}\big) \phi_{j_1}\big(x^{(k)}\big) \phi_{j_2}\big(x^{(k)}\big) \\
        &\hspace*{-1ex}\underset{\scriptscriptstyle m \rightarrow \infty}{\longrightarrow} \int \phi_{i_1}(x) \ts \phi_{i_2}(x) \ts \phi_{j_1}(x) \ts \phi_{j_2}(x) \ts \mathrm{d}\mu(x)
        = \innerprod{\phi_{i_1} \ts \phi_{i_2}}{\phi_{j_1} \ts \phi_{j_2}}_\mu.
\end{align*}
Analogously, we obtain
\begin{alignat*}{4}
    &\left[\frac{1}{m} Z_2 \ts Z_1^\top\right]_{\overline{i_1 i_2}\; j} && \underset{\scriptscriptstyle m \rightarrow \infty}{\longrightarrow} \innerprod{\phi_{i_1} \ts \phi_{i_2}}{\phi_{j}}_\mu,
    &~~\left[\frac{1}{m} Z_2 \ts \mathds{1}_m\right]_{\overline{i_1 i_2}} &\underset{\scriptscriptstyle m \rightarrow \infty}{\longrightarrow} \innerprod{\phi_{i_1} \ts \phi_{i_2}}{\mathds{1}}_\mu, \\
    &\left[\frac{1}{m} Z_1 \ts Z_1^\top\right]_{ij} &&\underset{\scriptscriptstyle m \rightarrow \infty}{\longrightarrow} \innerprod{\phi_i}{\phi_j}_\mu,
    &\left[\frac{1}{m} Z_1 \ts \mathds{1}_m\right]_{i}~~ &\underset{\scriptscriptstyle m \rightarrow \infty}{\longrightarrow} \innerprod{\phi_i}{\mathds{1}}_\mu,
\end{alignat*}
where $ \mathds{1}(x) \equiv 1 $ is the constant one function. For the right-hand side, we have
\begin{equation*}
    \left[\frac{1}{m} Z_2 \ts \dot{Z}_\ell^\top\right]_{\overline{i_1 i_2}\,} \underset{\scriptscriptstyle m \rightarrow \infty}{\longrightarrow} \innerprod{\phi_{i_1} \ts \phi_{i_2}}{\dot{z}_\ell}_\mu = \innerprod{\phi_{i_1} \ts \phi_{i_2}}{\nabla \phi_\ell \vdot F}_\mu,
\end{equation*}
and similarly
\begin{alignat*}{4}
    \left[\frac{1}{m} Z_1 \ts \dot{Z}_\ell^\top\right]_{i\,\,} && \underset{\scriptscriptstyle m \rightarrow \infty}{\longrightarrow} \innerprod{\phi_i}{\nabla \phi_\ell \vdot F}_\mu, ~~
    \frac{1}{m} \mathds{1}_m^\top \ts \dot{Z}_\ell^\top && \underset{\scriptscriptstyle m \rightarrow \infty}{\longrightarrow} \innerprod{\mathds{1}}{\nabla \phi_\ell \vdot F}_\mu.
\end{alignat*}
Overall, this results in the system of linear equations
\begin{equation*}
    \renewcommand*{\arraystretch}{1.3}
    \underbrace{
    \left[\,
    \begin{matrix}[c|c|c]
        \innerprod{\phi_{i_1} \ts \phi_{i_2}}{\phi_{j_1} \ts \phi_{j_2}}_\mu & \innerprod{\phi_{i_1} \ts \phi_{i_2}}{\phi_{j}}_\mu & \innerprod{\phi_{i_1} \ts \phi_{i_2}}{\mathds{1}}_\mu \\ \hline
        \innerprod{\phi_i}{\phi_{j_1} \ts \phi_{j_2}}_\mu & \innerprod{\phi_i}{\phi_j}_\mu & \innerprod{\phi_i}{\mathds{1}}_\mu \\ \hline
        \innerprod{\mathds{1}}{\phi_{j_1} \ts \phi_{j_2}}_\mu & \innerprod{\mathds{1}}{\phi_j}_\mu & \innerprod{\mathds{1}}{\mathds{1}}_\mu
    \end{matrix}\,\right]}_{R^*}
    \underbrace{
    \begin{bmatrix}
        A_\ell^\top \\
        B_\ell^\top \\
        C_\ell
    \end{bmatrix}}_{v_\ell^*}
    =
    \underbrace{
    \begin{bmatrix}
        \innerprod{\phi_{i_1} \ts \phi_{i_2}}{\nabla \phi_\ell \vdot F} \\
        \innerprod{\phi_i}{\nabla \phi_\ell \vdot F} \\
        \innerprod{\mathds{1}}{\nabla \phi_\ell \vdot F}
    \end{bmatrix}}_{s_\ell^*}.
\end{equation*}
In the infinite data limit, we can thus view QENDy as $ N $ continuous best approximations of $ \nabla \phi_\ell \vdot F $, where $ \ell = 1, \dots, N $, with respect to the $ \mu $-weighted inner product using the (linearly dependent) basis functions
\begin{equation*}
    \{ \phi_{i_1} \ts \phi_{i_2} \}_{i_1, i_2=1}^N \cup \{ \phi_i \}_{i=1}^N \cup \{\mathds{1} \},
\end{equation*}
which we will call the \emph{quadratically augmented basis} in what follows. The corresponding augmented feature map will be denoted by $ \overline{\phi} \colon \R^n \to \R^{N^2 + N + 1} $. QENDy itself is a discrete best approximation. More details about discrete and continuous best approximation problems can be found in Appendix~\ref{app:Best approximation}.

\begin{example} \label{ex:convergence}
In order to corroborate the above convergence results, let us consider again the damped pendulum and generate $ m $ uniformly distributed (since this simplifies the convergence analysis) training data points $ x^{(k)} $ in $ [-1, 1] \times [-1, 1] $, for which we then compute the exact derivatives $ \dot{x}^{(k)} $. Using the feature space mapping described in Example~\ref{ex:guiding examples}, we compute the matrix $ R $ and the vectors $ s_\ell $ using QENDy and $ R^* $ and $ s_\ell^* $ by computing the required integrals analytically. We then compare the entries element-wise and calculate the mean difference averaged over 100 runs. Additionally, we simulate the identified models for $ m = 6 $, $ m = 8 $, and $ m = 10 $ and compare the trajectories with the true solution. The results, shown in Figure~\ref{fig:convergence}, illustrate that QENDy indeed converges to the best approximation and that the identified models are good approximations of the true dynamics, even for small values of $ m $. \exampleSymbol

\begin{figure}
    \centering
    \begin{minipage}{0.45\linewidth}
        \centering
        \subfiguretitle{~~(a)}
        \vspace*{0.5ex}
        \includegraphics[width=0.95\linewidth]{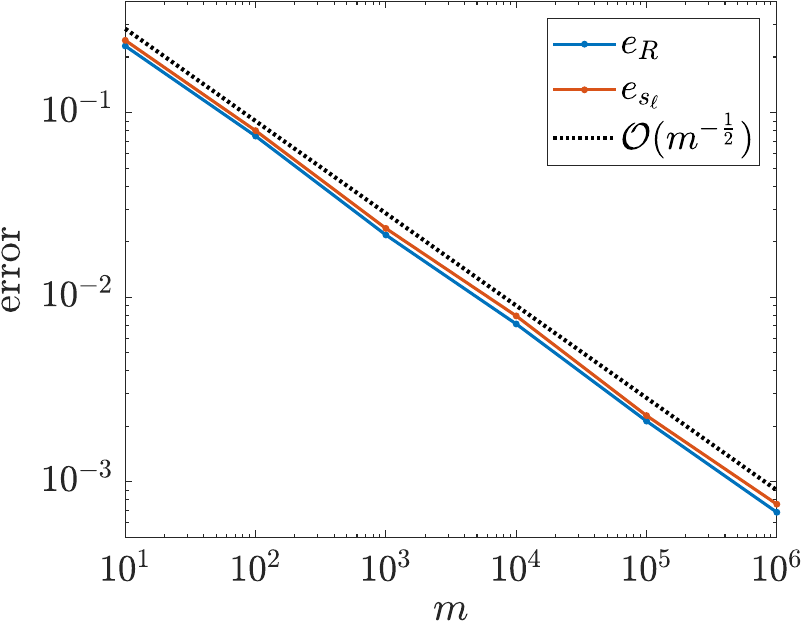}
    \end{minipage}
    \begin{minipage}{0.5\linewidth}
        \centering
        \subfiguretitle{(b)\hspace*{6em}}
        \vspace*{0.9ex}
        \includegraphics[width=0.995\linewidth]{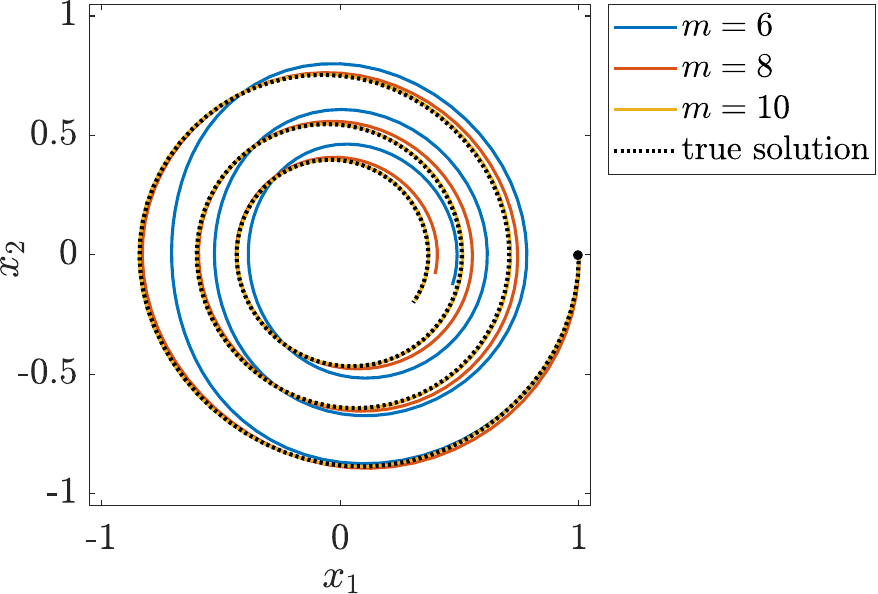}
    \end{minipage}
    \caption{(a) Average difference between the entries of $ R $ and $ R^* $ (denoted by $ e_R $) as well as $ s_\ell $ and $ s_\ell^* $ (denoted by $ e_{s_\ell} $) as a function of the number of training data points $ m $, illustrating the expected $ \mathcal{O}(m^{-\frac{1}{2}}) $ Monte Carlo convergence. (b) Simulation of the learned quadratic models for $ m = 6 $, $ m = 8 $, and $ m = 10 $. The black dot represents the initial condition $ x = [1, 0]^\top $.}
    \label{fig:convergence}
\end{figure}

\end{example}

The continuous best approximation is not required to correctly identify the governing equations. In fact, we typically just need a few independent observations in order to estimate $ A $, $ B $, and $ C $, see Remark~\ref{rem:regression problem}. A more detailed analysis of the prediction error can be found in Example~\ref{ex:QENDy vs. SINDy}.

\subsection{Regularization and stability}

For autonomous quadratic systems, i.e., $C = 0$ in \eqref{eq:quadratic ODE}, the stability radius of the equilibrium at the origin  is inversely proportional to the norm of the matrix $A$. This result relies on Lyapunov stability theory and requires that the matrix $B$ is Hurwitz stable, i.e., all the eigenvalues of $B$ have negative real parts \cite{Kra21, SKP23}. The latter condition ensures that the origin is locally asymptotically stable in the first place. The paper \cite{SKP23} extends this result to cubic systems and leverages it to propose a regularization of the operator inference method \cite{PW16} for non-intrusive reduced-order modeling that promotes models with a large stability radius by penalizing the Frobenius norms of the learned quadratic and cubic operators. This regularization does, however, not guarantee that the learned matrix $B$ is Hurwitz stable. Instead, it relies on an eigen-reflection post-processing step to enforce Hurwitz stability.

Hurwitz stability of the matrix $B$ can be incorporated into the optimization process to learn quadratic models that are stable by design by exploiting the fact that any Hurwitz stable matrix $B$ can be written as a product of generalized negative definite matrix $P$ and a symmetric positive definite matrix $Q$ \cite{DP98}. This idea is leveraged in \cite{GPB23} and applied to the operator inference method and SINDy, where in particular the parametrization
\begin{equation*}
	P =  \left(J - J^T\right) - R \ts R^\top, \quad  \text{and } \quad Q = L \ts L^\top
\end{equation*}
is used. Here, $ J $, $ R $, and $ L $ are general matrices of the same dimensions as $ B $, which ensures that $ Q $ is semi-positive definite only. It is important to note, however, that the resulting optimization problem is nonlinear, non-convex, and can only be solved using iterative methods, whereas the $ \norm{A}_F $-penalized operator inference optimization problem has an analytic solution.

In this work, we follow the same approach as in \cite{SKP23} to promote stability in models learned by QENDy. In particular, we modify the loss function \eqref{eq:loss function} as follows:
	\begin{equation} \label{eq:loss function_regularized}
		L(A, B, C) = \norm{\dot{Z} - A \ts Z_2 - B \ts Z_1 - C \ts \mathds{1}_m^\top}_F^2 + \lambda \norm{A}_F^2,
	\end{equation}
where $\lambda \geq 0$ is a regularization parameter. Following the same approach as in Section \ref{sec:derivation_QENDy}, we obtain the system of linear equations
	\begin{equation} \label{eq:QENDy_regularized}
		\renewcommand*{\arraystretch}{1.3}
		\underbrace{
			\left[\,
			\begin{matrix}[c|c|c]
				Z_2 \ts Z_2^\top + \lambda I_N & Z_2 \ts Z_1^\top & Z_2 \ts \mathds{1}_m \\ \hline
				Z_1 \ts Z_2^\top & Z_1 \ts Z_1^\top & Z_1 \ts \mathds{1}_m \\ \hline
				\mathds{1}_m^\top \ts Z_2^\top & \mathds{1}_m^\top \ts Z_1^\top & m
			\end{matrix}\,\right]}_{R_\lambda}
		\underbrace{
			\begin{bmatrix}
				A_\ell^\top \\
				B_\ell^\top \\
				C_\ell
		\end{bmatrix}}_{v_\ell}
		=
		\underbrace{
			\begin{bmatrix}
				Z_2 \ts \dot{Z}_\ell^\top \\
				Z_1 \ts \dot{Z}_\ell^\top \\
				\mathds{1}_m^\top \ts \dot{Z}_\ell^\top
		\end{bmatrix}}_{s_\ell},
	\end{equation}
which we solve to obtain  $ v_\ell = R_\lambda^+ s_\ell$.

It is important to note that although we do not ensure that $B$ is Hurwitz stable, we observed an improvement in stability using the regularization in \eqref{eq:QENDy_regularized} in cases where the dictionary is not expressive enough to represent the dynamics quadratically. Ensuring the stability of models learned from data, either quadratic or not, is an open problem and will be considered in future work.

\subsection{Comparison with SINDy}

Let us compare QENDy with the well-known method SINDy \cite{BPK16}. Although the original SINDy formulation was limited to ordinary differential equations, similar methods have been developed for stochastic differential equations \cite{BNC18} and partial differential equations \cite{RBPK17}. Given again training data of the form $ \{ (x^{(k)}, \dot{x}^{(k)}) \}_{k=1}^m $ and a set of basis functions $ \{ \phi_1, \dots, \phi_N \} $, we define
\begin{align*}
    \Phi_X &=
    \begin{bmatrix}
        \phi(x^{(1)}) & \phi(x^{(2)}) & \dots & \phi(x^{(m)})
    \end{bmatrix}
    \in \R^{N \times m}, \\
    \dot{X} &=
    \begin{bmatrix}
        \dot{x}^{(1)} & \dot{x}^{(2)} & \dots & \dot{x}^{(m)}
    \end{bmatrix}
    \in \R^{n \times m},
\end{align*}
i.e., $ \Phi_X = Z_1 $. The SINDy loss function is then given by
\begin{equation*}
    L(\Xi) = \norm{\dot{X} - \Xi \ts \Phi_X}_F^2,
\end{equation*}
where $ \Xi \in \R^{n \times N} $. In addition to minimizing the regression error, it is also possible to penalize non-sparse solutions $ \Xi $ in order to find models that are as simple as possible. The identified system is then defined by the optimal solution $ \Xi $ that minimizes the loss function, i.e.,
\begin{equation*}
    \dot{x} = \Xi \ts \phi(x).
\end{equation*}
Neglecting the sparsity constraints and following the same steps as above, we obtain the system of equations
\begin{equation*}
    \Phi_X \ts \Phi_X^\top \Xi^\top = \Phi_X \ts \dot{X}^\top,
\end{equation*}
which we decouple into $ n $ problems of the form
\begin{equation*}
    \Phi_X \ts \Phi_X^\top \Xi_\ell^\top = \Phi_X \ts \dot{X}_\ell^\top,
\end{equation*}
where $ \Xi_\ell $ and $ \dot{X}_\ell $ denote the $ \ell $th rows of $ \Xi $ and $ \dot{X} $, respectively. Dividing again both sides by the number of data points $ m $ and letting $ m $ go to infinity yields
\begin{align*}
    \left[\frac{1}{m} \Phi_X \ts \Phi_X^\top\right]_{ij} &\underset{\scriptscriptstyle m \rightarrow \infty}{\longrightarrow} \innerprod{\phi_i}{\phi_j}_\mu, \\
    \left[\frac{1}{m} \Phi_X \ts \dot{X}_\ell^\top\right]_{i} &\underset{\scriptscriptstyle m \rightarrow \infty}{\longrightarrow} \innerprod{\phi_i}{F_\ell}_\mu.
\end{align*}
That is, in the infinite data limit we are solving the system of equations
\begin{equation*}
    \begin{bmatrix}
        \innerprod{\phi_i}{\phi_j}_\mu
    \end{bmatrix}
    \Xi_\ell^\top
    =
    \begin{bmatrix}
        \innerprod{\phi_i}{F_\ell \vphantom{|_j}}_\mu
    \end{bmatrix}
\end{equation*}
for $ \ell = 1, \dots, n $ in the least-squares sense. This illustrates that SINDy is a data-driven best approximation of $ F $ using the basis functions $ \{ \phi_i \}_{i=1}^N $.

\begin{remark}
Using properties of the pseudoinverse, we could have also directly written
\begin{equation*}
    \Xi = \dot{X} \ts \Phi_X^+ = \big(\dot{X} \Phi_X^\top\big) \big(\Phi_X \Psi_X^\top\big)^+,
\end{equation*}
which results in the same system of equations. Our goal, however, was to highlight similarities with the quadratic embedding. Directly computing the pseudoinverse of $ \Phi_X $ instead of solving the normal equation is more expensive, but also more accurate, see Remark \ref{rem:regression problem}. The computational complexity is $ \mathcal{O}\big(m \ts N^2\big) $.
\end{remark}

\begin{example} \label{ex:QENDy vs. SINDy}

\begin{figure}
    \centering
    \begin{minipage}{0.325\linewidth}
        \centering
        \subfiguretitle{~~(a)}
        \includegraphics[width=0.98\linewidth]{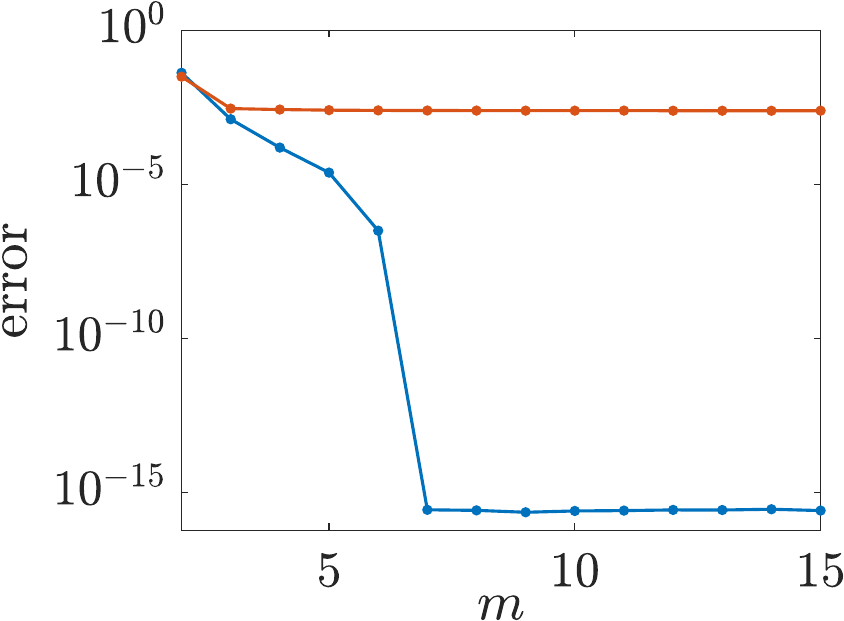}
    \end{minipage}
    \begin{minipage}{0.325\linewidth}
        \centering
        \subfiguretitle{~~(b)}
        \includegraphics[width=0.98\linewidth]{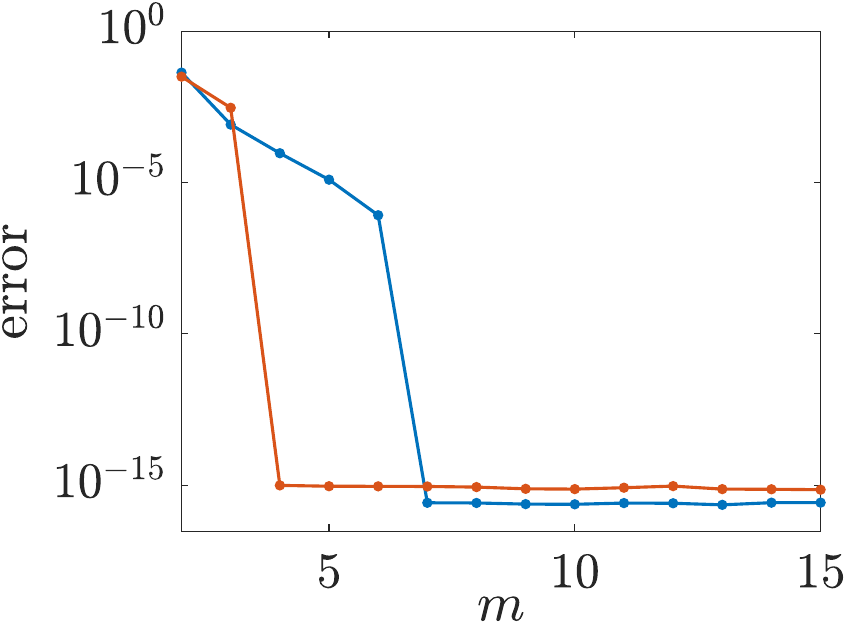}
    \end{minipage}
    \begin{minipage}{0.325\linewidth}
        \centering
        \subfiguretitle{~~(c)}
        \includegraphics[width=0.98\linewidth]{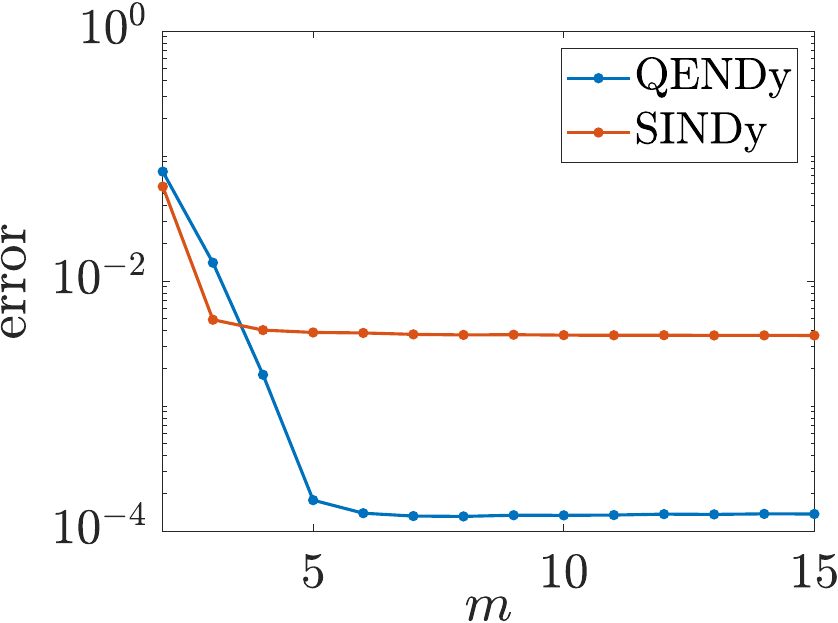}
    \end{minipage}
    \caption{Comparison of QENDy and SINDy using (a) the basis functions defined in Example~\ref{ex:guiding examples}, (b) an extended dictionary comprising four basis functions so that both methods correctly identify the system, (c) monomials of order up to two.}
    \label{fig:QENDy vs. SINDy}
\end{figure}

Let us consider the simple systems defined in Example~\ref{ex:guiding examples}:
\begin{enumerate}[leftmargin=3.5ex, itemsep=0ex, topsep=0.5ex, label=\roman*)]
\item SINDy would not be able to find the governing equations of the rational system using the three selected basis functions since the right-hand side cannot be represented as a linear combination of these functions. The approximation error for different values of $ m $, defined to be the average difference between the predicted derivative and the true derivative evaluated in $ 100 $ test data points, is shown in Figure~\ref{fig:QENDy vs. SINDy}\ts(a). While QENDy identifies the correct dynamics (up to machine precision), SINDy only approximates the dynamics. If we add the function $ \frac{x^2}{(1+x)^2} $ to the dictionary, then both QENDy and SINDy correctly identify the system as shown in Figure~\ref{fig:QENDy vs. SINDy}\ts(b). Choosing a dictionary comprising only monomials of order up to two, both methods are not able to learn the true dynamics, see Figure~\ref{fig:QENDy vs. SINDy}\ts(c). In this case, QENDy results in smaller approximation errors due to the higher-dimensional feature space, provided we have enough observations to correctly estimate all parameters.
\item The governing equations of the nonlinear pendulum can be identified using SINDy with the chosen dictionary. In this case, we obtain
\begin{equation}
    \dot{x} =
    \underbrace{
    \begin{bmatrix}
        0 & 1 & 0  & 0 \\
        0 & -0.1 & -1 & 0
    \end{bmatrix}}_{\Xi}
    \begin{bmatrix}
        x_1 \\
        x_2 \\
        \sin(x_1) \\
        \cos(x_1)
    \end{bmatrix}. \tag*{\exampleSymbol}
\end{equation}
\end{enumerate}
\end{example}

\subsection{Comparison with gEDMD}

There are also closely related methods for learning the governing equations that are based on the Koopman generator \cite{MauGon16, KNPNCS20}. We will now compare QENDy and gEDMD \cite{KNPNCS20}, which aims at approximating the Koopman generator associated with ordinary or stochastic differential equations. Let $ \dot{\phi}(x) = (\mathcal{L} \phi)(x) = J(x) F(x) $, where the Koopman generator associated with \eqref{eq:ODE} is again applied component-wise. We define the matrices
\begin{align*}
    \Phi_X &=
    \begin{bmatrix}
        \phi(x^{(1)}) & \phi(x^{(2)}) & \dots  & \phi(x^{(m)})
    \end{bmatrix}
    \in \R^{N \times m}, \\
    \dot{\Phi}_X &=
    \begin{bmatrix}
        \dot{\phi}(x^{(1)}) & \dot{\phi}(x^{(2)}) & \dots & \dot{\phi}(x^{(m)}) \\
    \end{bmatrix}
    \in \R^{N \times m}.
\end{align*}
The loss function is in this case defined by
\begin{equation*}
    L(\Theta) = \norm{\dot{\Phi}_X - \Theta \ts \Phi_X}_F^2.
\end{equation*}
Setting the derivative of the loss function with respect to the matrix $ \Theta $ to zero results in the system of equations
\begin{equation*}
    \Phi_X \ts \Phi_X^\top \Theta^\top = \Phi_X \ts \dot{\Phi}_X^\top.
\end{equation*}
In the infinite data limit, we obtain
\begin{align*}
    \left[\frac{1}{m} \Phi_X \ts \Phi_X^\top\right]_{ij} &\underset{\scriptscriptstyle m \rightarrow \infty}{\longrightarrow} \innerprod{\phi_i}{\phi_j}_\mu, \\
    \left[\frac{1}{m} \Phi_X \ts \dot{\Phi}_X^\top\right]_{ij} &\underset{\scriptscriptstyle m \rightarrow \infty}{\longrightarrow} \innerprod{\phi_i}{\mathcal{L} \phi_j}_\mu,
\end{align*}
so that $ \Theta^\top $ can be regarded as the matrix representation of the Koopman generator projected onto the space spanned by the basis functions $ \phi $. In other words, gEDMD computes a Galerkin approximation of the operator $ \mathcal{L} $, where the required integrals are estimated from data. Given a function $ f(x) = \alpha^\top \phi(x) $, we have
\begin{equation*}
    \mathcal{L} f(x) \approx \alpha^\top \Theta \ts \phi(x).
\end{equation*}
For the full-state observable $ g(x) = x = G \ts \phi(x) $, this implies
\begin{equation*}
    \dot{x} = F(x) = \mathcal{L} g(x) \approx G \ts \Theta \ts \phi(x) = \Xi \ts \phi(x),
\end{equation*}
which shows that gEDMD can be used for system identification, see \cite{KNPNCS20} for details. However, gEDMD also works for stochastic differential equations and allows us to compute spectral properties of the Koopman generator and to decompose the right-hand side $ F $ into eigenvalues, eigenfunctions, and associated modes. To demonstrate that QENDy and gEDMD are related as well, we now approximate the Koopman generator using the augmented basis given by $ \overline{\phi} $. It then holds that $ R = \overline{\Phi}_X \ts \overline{\Phi}_X^\top $ and $ s_\ell $ is the $ (N^2+\ell) $th column of the matrix $ \overline{\Phi}_X \ts \dot{\overline{\Phi}}{\vphantom{\overline{\Phi}}}_X^\top $. The Koopman generator approximation with respect to the quadratically augmented basis contains more information about the system, namely the action of the Koopman generator applied to products of basis functions.

\subsection{Comparison with quadratic embeddings using deep learning} \label{sec:dl-quadembed}

QENDy uses a dictionary containing a predefined set of basis functions to map a nonlinear dynamical system to an embedding space where the dynamics can be expressed in a quadratic form, either exactly or approximately. A different approach, proposed in \cite{GB24}, is to define an embedding map, $ x \mapsto z = \phi_{w_\text{enc}}(x) $, using an encoder neural network and to reconstruct $ x $ from $ z $ using a decoder neural network, $ z \mapsto x = \psi_{w_\text{dec}}(z) $. The goal is then to optimize the parameters $ w_\text{enc} $ and $ w_\text{dec} $ of the encoder and decoder concurrently with $ A $, $ B $, and $ C $ such that $ z $ satisfies both \eqref{eq:quadratic ODE} and the encoder--decoder reconstruction property $\psi_{w_\text{dec}}\left( \phi_{w_\text{enc}}(x) \right) = x$. This leads to the following loss functions that need to be optimized simultaneously:
\begin{enumerate}[leftmargin=3.5ex, itemsep=0ex, topsep=0.5ex, label=\roman*)]
\item Encoder--decoder reconstruction:
\begin{equation} \label{eq:loss 1 quadembed}
	\mathcal{L}_\text{encdec} = \sum_{k=1}^{m}\left\| x^{(k)} - \psi_{w_\text{dec}}\big(\phi_{w_\text{enc}}\big( x^{(k)}\big) \big) \right\|^2.
\end{equation}
\item Quadratic dynamics for $ z $ given $ \dot{x} $:
\begin{equation} \label{eq:loss 2 quadembed}
	\mathcal{L}_{\dot{z}\dot{x}} = \sum_{k=1}^{m}\left\| \dot{z}^{(k)} - A \big(z^{(k)} \otimes z^{(k)}\big) - B \ts z^{(k)} - C \right\|^2,
\end{equation}
where $z^{(k)} = \phi_{w_\text{enc}}\left(x^{(k)}\right)$,  $\dot{z}^{(k)} = J_{\phi_{w_\text{enc}}} \left(x^{(k)}\right) \, \dot{x}^{(k)}$, and $J_{\phi_{w_\text{enc}}} \left(x^{(k)}\right)$ is the Jacobian matrix of the encoder map $\phi_{w_\text{enc}}$ at $x^{(k)}$ evaluated via automatic differentiation.
\item Recovery of $\dot{x}$ from the quadratic expression of $\dot{z}$:
\begin{equation} \label{eq:loss 3 quadembed}
	\mathcal{L}_{\dot{x}\dot{z}} = \sum_{k=1}^{m}\left\| \dot{x}^{(k)} - J_{\psi_{w_\text{dec}}} \big(z^{(k)}\big) \big( A \big(z^{(k)} \otimes z^{(k)}\big) + B \ts z^{(k)} + C \big) \right\|^2.
\end{equation}
\end{enumerate}
In other words, $ \mathcal{L}_\text{encdec} $ ensures that the encoder and decoder networks are consistent, i.e., the input to the encoder can be recovered by passing the associated output through the decoder, $ \mathcal{L}_{\dot{z}\dot{x}} $ enforces the quadratic structure of the system implied by the encoder map, and $ \mathcal{L}_{\dot{x}\dot{z}} $ is a requirement that the original system can be recovered from this quadratic representation by applying the decoder map. Overall, the goal is to solve the optimization problem
\begin{equation*}
	A, B, C, w_\text{enc}, w_\text{dec} = \underset{\hat{A}, \hat{B}, \hat{C}, \hat{w}_\text{enc}, \hat{w}_\text{dec}}{\argmin} L\left(\hat{A} ,\, \hat{B},\, \hat{C},\, \hat{w}_\text{enc},\, \hat{w}_\text{dec}\right),
\end{equation*}
where $ L = \lambda_1\, \mathcal{L}_\text{encdec} + \lambda_2\, \mathcal{L}_{\dot{x}\dot{z}} + \lambda_3\, \mathcal{L}_{\dot{z}\dot{x}} $ and $\lambda_i, \: i \in \{1,2,3\} $, are non-negative hyperparameters that need to be tuned. The norm $\norm{\cdot} $ used in \eqref{eq:loss 1 quadembed}--\eqref{eq:loss 3 quadembed} is defined to be $ \norm{\cdot} = \norm{\cdot}_2 + \kappa \norm{\cdot}_1 $, where $ \kappa $ is typically between $ 0 $ and $ 1 $. It is important to note that in the implementation of the method, the authors of \cite{GB24} use iterative hard thresholding to promote sparsity in the learned matrices $ A $ and $ B $. Therefore, the method requires tuning five hyperparameters (the $ \lambda_i $, $ \kappa $, and the hard thresholding level) that all have an effect on its performance. The structure and size of the encoder and decoder networks, and optimization-related hyperparameters such as the learning rate and the batch size also play a pivotal role on the performance of the method in practice.

\subsection{Summary}

The above analysis shows that QENDy and SINDy are closely related, the quadratic embedding just generates a higher-dimensional feature space by using not only the basis functions themselves but also products of basis functions. However, QENDy approximates $ \nabla \phi_\ell \vdot F $ for $ \ell = 1, \dots, N $, from which we then extract $ F $ using~\eqref{eq:projection of quadratic ODE}, whereas SINDy approximates $ F_\ell $ for $ \ell = 1, \dots, n $. Furthermore, QENDy learns a quadratic system, while the SINDy model still requires evaluating the typically nonlinear basis functions. The deep learning approach identifies not only $ A $, $ B $, and $ C $, but also the nonlinear transformation itself, which makes it more flexible, but less interpretable. The learned models, however, do often not generalize to unseen data, i.e., they are suitable for interpolation but not necessarily extrapolation. QENDy, on the other hand, identifies a globally defined model whose accuracy depends strongly on the chosen dictionary.

\section{Numerical results}
\label{sec:Numerical results}

In this section, we will present numerical results for various benchmark problems.

\subsection{Modified Thomas attractor}

Consider the system
\begin{equation*}
    \begin{bmatrix}
        \dot{x}_1 \\
        \dot{x}_2 \\
        \dot{x}_3
    \end{bmatrix}
    =
    \begin{bmatrix}
        \sin(x_2) - \alpha \ts x_1 - \beta \ts x_2 \ts \cos(x_1) \\
        \sin(x_3) - \alpha \ts x_2 - \beta \ts x_3 \ts \cos(x_2) \\
        \sin(x_1) - \alpha \ts x_3 - \beta \ts x_1 \ts \cos(x_3)
    \end{bmatrix},
\end{equation*}
where $ \alpha $ and $ \beta $ are parameters. If $ \beta = 0 $, we obtain the original Thomas attractor \cite{Thomas99} as a special case. We consider two different test cases: (a)~$ \alpha = 0.2 $ and $ \beta = 0 $ and (b)~$ \alpha = 0.25 $ and $ \beta = 0.15 $. Trajectories for both parameter settings are shown in Figure~\ref{fig:Thomas attractor}. We then define
\begin{equation} \label{eq:Thomas embedding}
    \renewcommand*{\arraystretch}{1.05}
    z =
    \begin{bmatrix}
        x_1 \\
        x_2 \\
        x_3 \\
        \sin(x_1) \\
        \sin(x_2) \\
        \sin(x_3) \\
        \cos(x_1) \\
        \cos(x_2) \\
        \cos(x_3)
    \end{bmatrix}
    \implies
    \dot{z} =
    \begin{bmatrix}
        z_5 - \alpha \ts z_1 - \beta \ts z_2 \ts z_7 \\
        z_6 - \alpha \ts z_2 - \beta \ts z_3 \ts z_8 \\
        z_4 - \alpha \ts z_3 - \beta \ts z_1 \ts z_9 \\
        z_5 \ts z_7 - \alpha \ts z_1 \ts z_7 - \beta \ts z_2 \ts z_7^2 \\
        z_6 \ts z_8 - \alpha \ts z_2 \ts z_8 - \beta \ts z_3 \ts z_8^2 \\
        z_4 \ts z_9 - \alpha \ts z_3 \ts z_9 - \beta \ts z_1 \ts z_9^2 \\
        -z_4 \ts z_5 + \alpha \ts z_1 \ts z_4 + \beta \ts z_2 \ts z_4 \ts z_7 \\
        -z_5 \ts z_6 + \alpha \ts z_2 \ts z_5 + \beta \ts z_3 \ts z_5 \ts z_8 \\
        -z_4 \ts z_6 + \alpha \ts z_3 \ts z_6 + \beta \ts z_1 \ts z_6 \ts z_9
    \end{bmatrix}.
\end{equation}
That is, if $ \beta = 0 $, the resulting system is quadratic in the embedding space. If, on the other hand, $ \beta \ne 0 $, the first three terms are still quadratic, but the remaining equations are of order three and would require introducing additional basis functions, see Appendix~\ref{app:Modified Thomas attractor}.

\begin{figure}
    \centering
    \begin{minipage}{0.45\linewidth}
        \centering
        \subfiguretitle{(a)}
        \vspace*{0.5ex}
        \includegraphics[width=\linewidth]{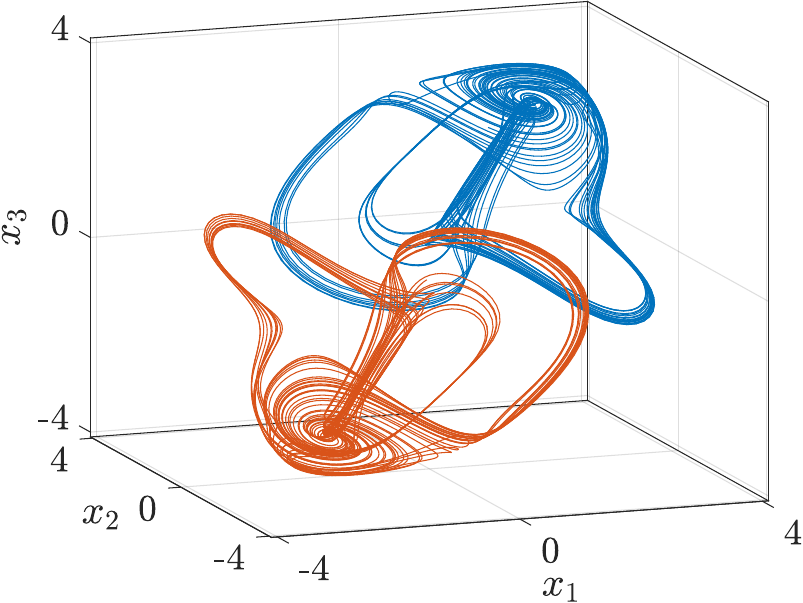}
    \end{minipage}
    \begin{minipage}{0.45\linewidth}
        \centering
        \subfiguretitle{(b)}
        \vspace*{0.5ex}
        \includegraphics[width=\linewidth]{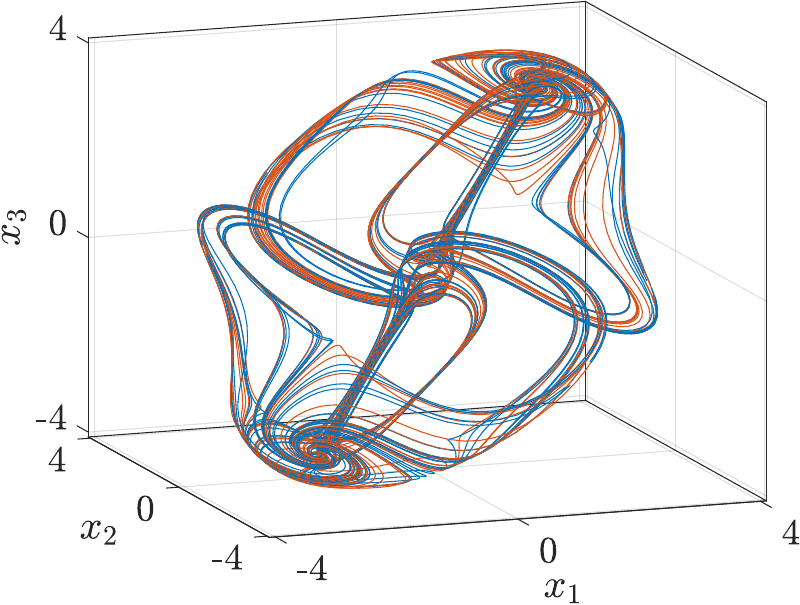}
    \end{minipage}
    \caption{Simulation of the modified Thomas system with (a) $ \alpha = 0.2 $ and $ \beta = 0 $ and (b) $ \alpha = 0.25 $ and $ \beta = 0.15 $ using two randomly generated initial conditions.}
    \label{fig:Thomas attractor}
\end{figure}

\paragraph{First test case.}

In order to learn the quadratic embedding of the dynamics, we generate one long trajectory consisting of 1000 training data points by simulating the Thomas system with initial condition $ x_0 = [1, -1, 0]^\top $ for $ t \in [0, 100] $. We then compute the corresponding time derivatives and apply QENDy. The resulting matrices $ A $ and $ B $ are sparse and the vector $ C $ is zero as shown in Figure~\ref{fig:Thomas system}\ts(a).  We simulate the quadratic model identified by QENDy using the initial condition $ x_0 = [0, 1, 1]^\top $ and compare the trajectory with the corresponding SINDy model and the true dynamics. The numerical results are shown in Figure~\ref{fig:Thomas system}\ts(c). QENDy successfully learns the correct quadratic embedding and the identified model is a very good approximation of the original dynamics. However, since the system is chaotic, even small numerical errors will eventually lead to large deviations. SINDy also works in this case as the right-hand side can be written as a linear combination of the basis functions, provided that $ \beta = 0 $.

\paragraph{Second test case.}

Although the right-hand side of the modified Thomas system can be written in terms of the quadratically augmented basis, the chosen basis functions are not sufficient for recasting the system as a quadratic model as shown in \eqref{eq:Thomas embedding}. Applying again QENDy, the matrices $ A $ and $ B $ and the vector $ C $ are not sparse anymore, see Figure~\ref{fig:Thomas system}\ts(b). Also SINDy with the chosen dictionary fails to correctly identify the system. Both methods nevertheless allow us to learn approximations of the true dynamics as illustrated in Figure~\ref{fig:Thomas system}\ts(d). We could in this case define
\begin{equation*}
    \dot{x} = G \ts A (\phi(x) \otimes \phi(x)) + G \ts B \ts \phi(x) + G \ts C,
\end{equation*}
which would be the correct model. This can be viewed as a combination of QENDy and SINDy, where we use the quadratically augmented basis to learn the right-hand side instead of a quadratic embedding. In order to obtain a valid quadratic embedding, we would have to choose a larger dictionary.

\begin{remark}
Additionally, we also applied the deep learning approach to both test cases. For the encoder mapping, we use a neural network with three fully-connected hidden layers of $64$ units each and ELU activation functions. We set the embedding dimension to four and choose $ \lambda_1 = 0.05 $,  $ \lambda_2 = 0.1 $, $ \lambda_3 = 1 $, and $ \kappa = 0 $. These settings were chosen after manual hyperparameter tuning. Following the original implementation published along with \cite{GB24}, we use a linear mapping for the decoder. Despite extensive hyperparameter tuning, the deep learning approach did not yield a good approximation for either test case and we therefore do not show the results.
\end{remark}

\begin{figure}
    \centering
    \begin{minipage}{\linewidth}
        \centering
        \subfiguretitle{(a)}
        \vspace*{0.5ex}
        \includegraphics[width=\linewidth]{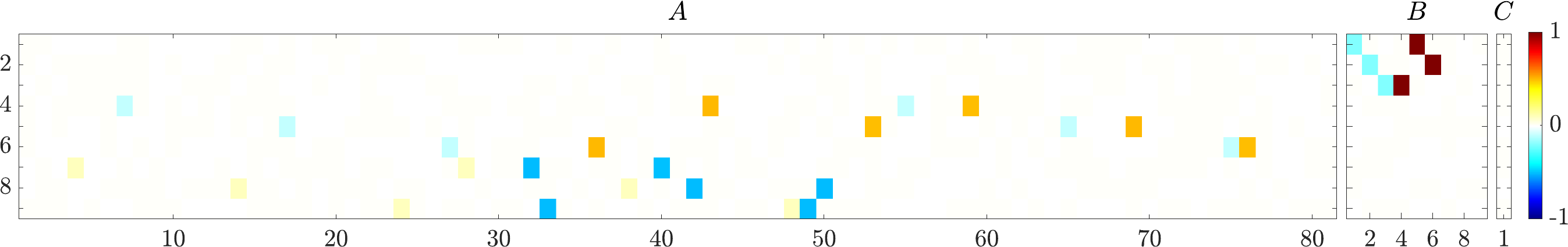}
    \end{minipage} \\[1ex]
    \begin{minipage}{\linewidth}
        \centering
        \subfiguretitle{(b)}
        \vspace*{0.5ex}
        \includegraphics[width=\linewidth]{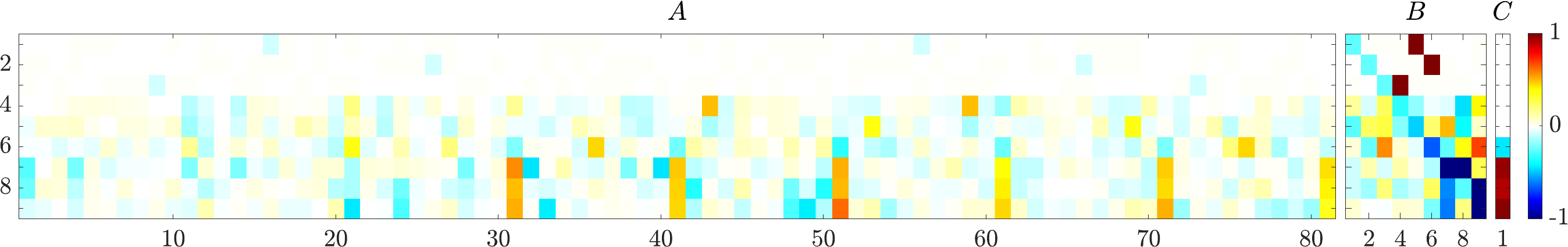}
    \end{minipage} \\[2ex]
    \begin{minipage}[t]{0.431\linewidth}
        \centering
        \subfiguretitle{(c)}
        \includegraphics[width=\linewidth]{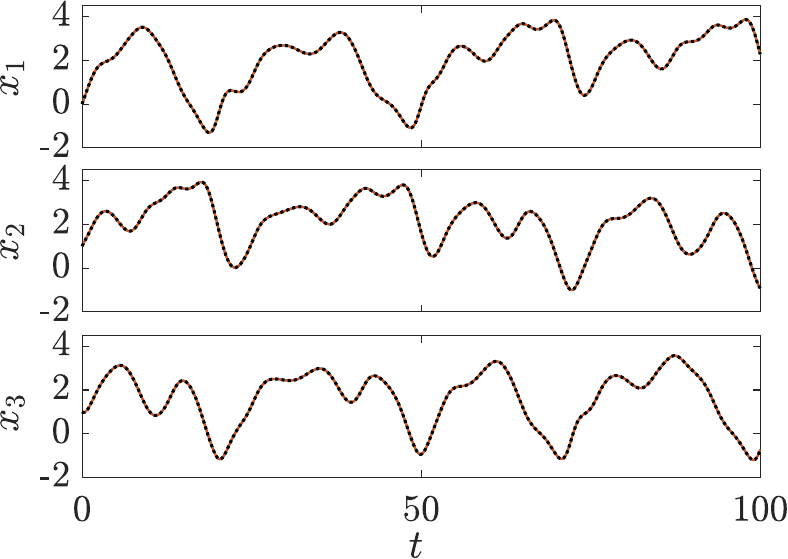}
    \end{minipage}
    \begin{minipage}[t]{0.42\linewidth}
        \centering
        \subfiguretitle{(d)}
        \vspace{0.2ex}
        \includegraphics[width=\linewidth]{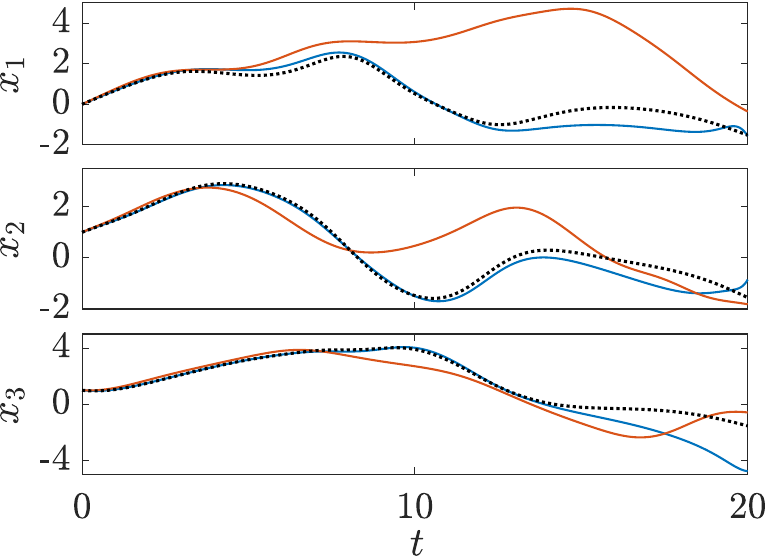} % x_0 = [-2, 0.2, 1.1]
    \end{minipage}
    \begin{minipage}[t]{0.11\linewidth}
        \vspace{0.85ex}
        \includegraphics[width=\linewidth]{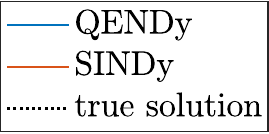}
    \end{minipage}
    \caption{(a) Identified system for $ \alpha = 0.2 $ and $ \beta = 0 $. (b) Identified system for $ \alpha = 0.25 $ and $ \beta = 0.15 $. Only the first three rows are sparse since the functions $ \dot{z}_4, \dots, \dot{z}_9 $ are not contained in the space spanned by the augmented basis. (c) Comparison of the QENDy model, the SINDy model, and the true dynamics for the first test case. Both methods correctly identify the governing equations and the numerical solutions coincide. (d) Comparison of the QENDy model, the SINDy model, and the true dynamics for the second test case. The learned models only allow us to make short-term predictions in this case. The QENDy approximation is slightly more accurate, but both models deviate or become unstable for larger $ t $.}
    \label{fig:Thomas system}
\end{figure}

\subsection{K\'arm\'an vortex street}

We now consider fluid flow past a cylinder in two dimensions at Reynolds number 100. This system is governed by partial differential equations. We discretize the domain using a regular $ 520 \times 180 $ grid and compute the vorticities in the grid points for $ t = 0, \dots, 499 $ as shown in Figure~\ref{fig:Karman vortex street}\ts(a)--(c). We then store the numerically computed vorticities in a three-dimensional array $ V \in \R^{180 \times 520 \times 500} $ and apply \emph{Principal Component Analysis} (PCA). There are three dominant principal components, shown in Figure~\ref{fig:Karman vortex street}\ts(d)--(f), followed by a spectral gap. That is, despite the high-dimensional state space caused by the grid discretization, the dynamics are approximately three-dimensional. By projecting the time-series data onto the dominant principal components, we obtain a three-dimensional dynamical system. It was shown in the Supporting Information supplementing \cite{BPK16} that SINDy successfully identifies an ordinary differential equation that approximates the dynamics in the three-dimensional principal component space. However, this is not sufficient to correctly predict the dynamics in the original state space. Mapping the learned dynamics back to the high-dimensional state space only allows us to identify the main periodic behavior, but small-scale features represented by the subsequent principal components will be neglected. Nonlinear model reduction techniques using, e.g., autoencoders might lead to more accurate lower-dimensional representations~\cite{EVNE20}.

Our goal is to compute a quadratic model using QENDy. We choose $ z = x $, i.e., the embedding space is the original projected principal component space, and estimate the required time derivatives with the aid of finite difference approximations (using a forward difference for the first data point, a backward difference for the last data point, and central differences in between). We apply QENDy to the first 400 data points and then simulate the identified quadratic model for $ t \in [0, 500] $. Furthermore, we also apply SINDy and the deep learning approach proposed in \cite{GB24} using a neural network with two fully connected layers containing twelve units each for the encoder map and a linear mapping for the decoder. Based on extensive manual hyperparameter tuning, we set the embedding dimension to four, the loss weights to $ \lambda_1 = 1 $, $ \lambda_2 = 40 $, and $ \lambda_3 = 80 $, and $ \kappa = 0 $. The results for all methods are shown in Figure~\ref{fig:Karman vortex street}\ts(g) and Figure~\ref{fig:Karman vortex street}\ts(h). The learned quadratic models are good approximations of the projected dynamics and correctly predict the periodic behavior, even though we do not know the true underlying system. QENDy---using only linear basis functions---is in this case slightly more accurate than the neural network approximation. The QENDy and SINDy results, provided we use the quadratically augmented basis, are as expected almost identical.

\begin{figure}
    \definecolor{matlab1}{RGB}{0, 114, 189}
    \definecolor{matlab2}{RGB}{217, 83, 25}
    \definecolor{matlab3}{RGB}{237, 177, 32}
    \definecolor{matlab4}{RGB}{126, 47, 142}
    \newcommand{\cdash}[1]{\textcolor{#1}{\rule[0.55ex]{1em}{0.2ex}}}
    \newcommand{\ddash}[1]{\textcolor{#1}{\raisebox{0.5ex}{\hbox to 2.7ex{\leaders\hbox to 2pt{\hss . \hss}\hfil}}}}
    \centering
    \begin{minipage}{0.32\linewidth}
        \centering
        \subfiguretitle{(a) $ t = 100 $}
        \vspace*{0.5ex}
        \includegraphics[width=\linewidth]{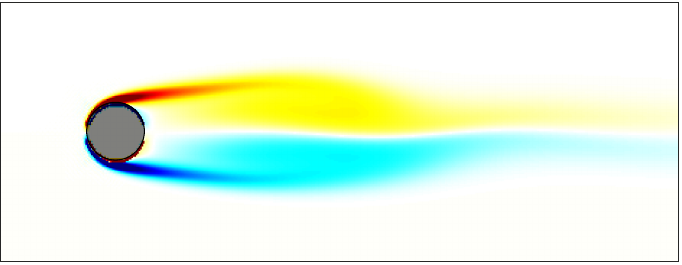}
    \end{minipage}
    \begin{minipage}{0.32\linewidth}
        \centering
        \subfiguretitle{(b) $ t = 200 $}
        \vspace*{0.5ex}
        \includegraphics[width=\linewidth]{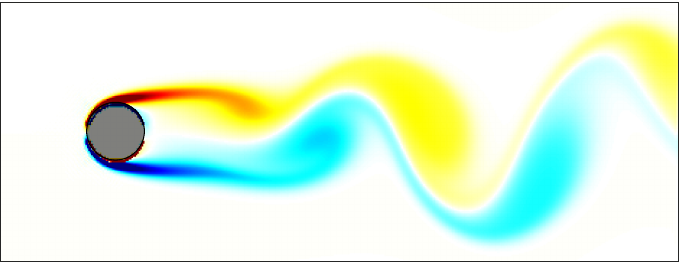}
    \end{minipage}
    \begin{minipage}{0.32\linewidth}
        \centering
        \subfiguretitle{(c) $ t = 300 $}
        \vspace*{0.5ex}
        \includegraphics[width=\linewidth]{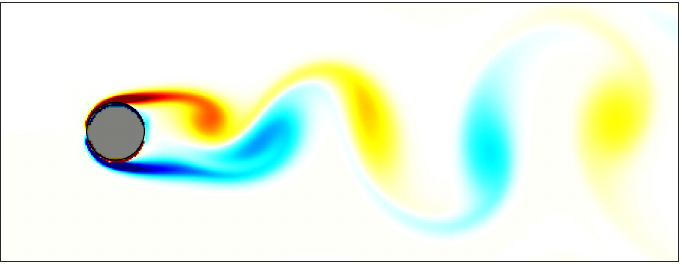}
    \end{minipage} \\[1.5ex]
    \begin{minipage}{0.32\linewidth}
        \centering
        \subfiguretitle{(d) PC 1}
        \vspace*{0.5ex}
        \includegraphics[width=\linewidth]{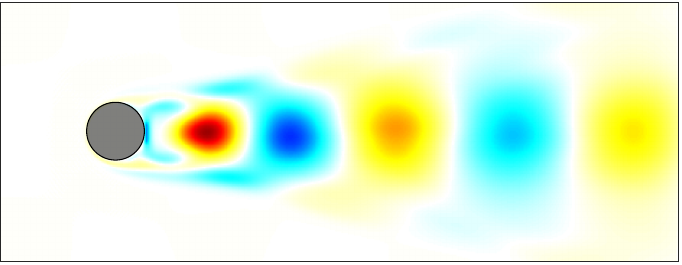}
    \end{minipage}
    \begin{minipage}{0.32\linewidth}
        \centering
        \subfiguretitle{(e) PC 2}
        \vspace*{0.5ex}
        \includegraphics[width=\linewidth]{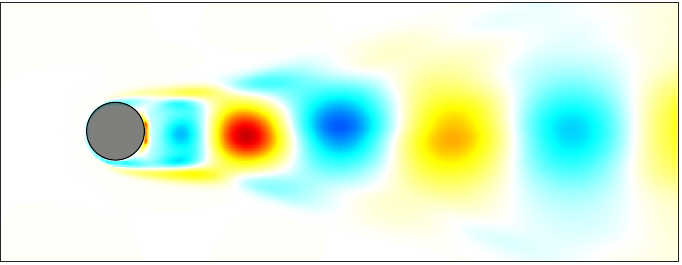}
    \end{minipage}
    \begin{minipage}{0.32\linewidth}
        \centering
        \subfiguretitle{(f) PC 3}
        \vspace*{0.5ex}
        \includegraphics[width=\linewidth]{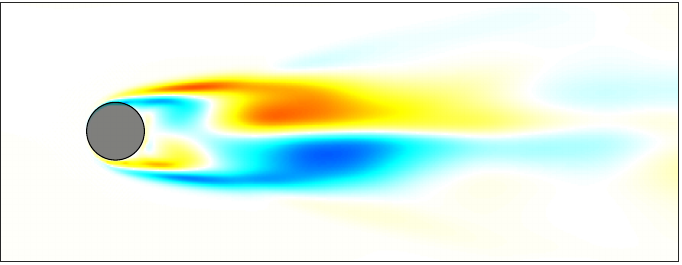}
    \end{minipage} \\[1.5ex]
    \begin{minipage}[t]{0.49\linewidth}
        \centering
        \subfiguretitle{(g)}
        \vspace*{0.5ex}
        \includegraphics[width=0.93\linewidth]{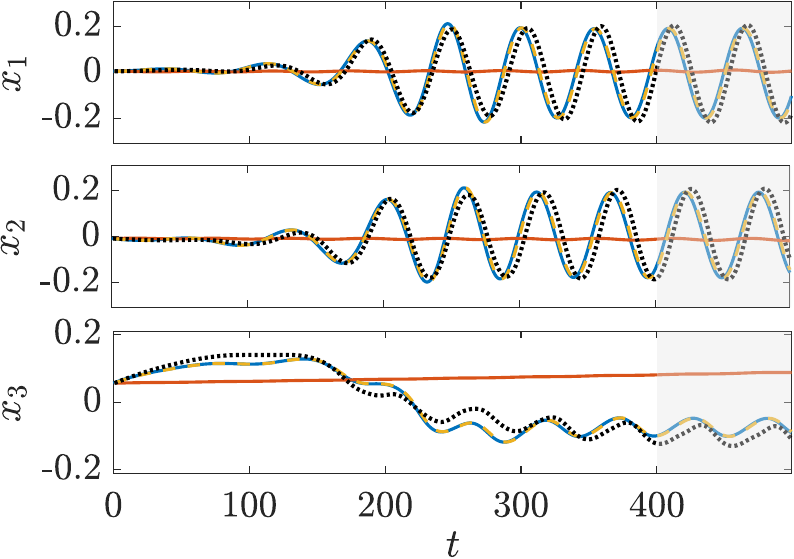} \\[1ex]
        \fbox{\scriptsize
        \cdash{matlab1} QENDy~~
        \cdash{matlab2} SINDy~~
        \cdash{matlab3} SINDy (aug)~~
        \ddash{black} true solution}
    \end{minipage}
    \begin{minipage}[t]{0.49\linewidth}
        \centering
        \subfiguretitle{(h)}
        \vspace*{0.5ex}
        \includegraphics[width=0.93\linewidth]{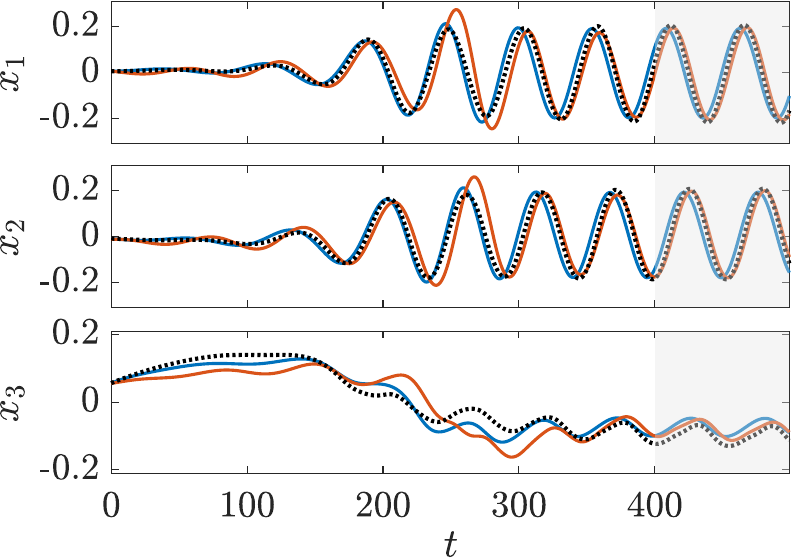} \\[1ex]
        \hspace{7ex}\fbox{\footnotesize
        \cdash{matlab1} QENDy~~
        \cdash{matlab2} DNN~~
        \ddash{black} true solution}
    \end{minipage}
    \caption{(a)--(c) Vorticities of the simulated K\'arm\'an vortex street at different times $ t $. After the transient phase, the system exhibits the well-known periodic vertex shedding. (d)--(f) The first three principal components. (g) Comparison of the QENDy model, the SINDy approach (using both the dictionary and the quadratically augmented dictionary), and the true dynamics. The area shaded in gray is the extrapolation domain, i.e., these points are not contained in the training data set. The trajectories converge to the limit cycle representing the periodic vortex shedding. (h)~Comparison of the QENDy model, the deep learning approach, and the true dynamics.}
    \label{fig:Karman vortex street}
\end{figure}

\section{Conclusion}
\label{sec:Conclusion}

We have derived a method called QENDy that allows us to learn quadratic embeddings of nonlinear dynamical systems from simulation or measurement data and can be viewed as a combination of SINDy \cite{BPK16} and the deep learning method proposed in \cite{GB24}. Given a fixed set of basis functions, the computation of the optimal embedding just requires solving systems of linear equations in the least squares sense. We have shown that in the infinite data limit we obtain the best approximation of the dynamics with respect to a weighted inner product that depends on the sampling of the training data points. While SINDy directly approximates the governing equations using the selected basis functions, QENDy uses a quadratically augmented basis. Our results also explain the purported superior performance of deep learning based quadratic embedding techniques compared to linear methods, which seems to be mostly due to the fact that these methods implicitly generate a larger and thus more expressive feature space. In general, QENDy will not result in a quadratic embedding that perfectly reproduces the dynamics---unless we have prior information about the terms appearing in the governing equations and choose a suitable dictionary---, but nevertheless computes the best approximation, which allows for short-term forecasting. The sparsity of the matrices $ A $ and $ B $ is typically a good indicator for the accuracy of the learned model. If the matrices are not sparse, this implies that almost all basis functions are required to approximate missing terms. That is, QENDy suffers from the same limitations as SINDy. Noisy data or unsuitable basis functions will result in inaccurate models and predictions.

Although the numerical results show that QENDy works well for simple benchmark problems, many open questions remain: High-dimensional systems will in general require a large set of basis functions, which significantly increases the number of parameters we have to estimate. A possible extension would be to develop a kernel-based variant of QENDy, which implicitly works in a potentially infinite-dimensional feature space rather than an explicitly constructed space spanned by a set of basis functions. This would allow us to circumvent the curse of dimensionality associated with the number of basis functions. The drawback is that then the dimensions of the matrices depend on the number of training data points. Alternatively, low-rank tensor approximations as proposed in \cite{GKES19} could be used to compress the data matrices. Furthermore, instead of using Tikhonov regularization, it would also be possible to use iterative hard thresholding or $ \ell_1 $ regression techniques to promote sparsity. Another open question is whether it is possible to incorporate physical constraints such as energy conservation laws or symplecticity and also to guarantee the stability of the learned model. A second-order accurate integrator tailored to quadratic systems has been proposed in \cite{KL97}, the question is whether the method is more robust than conventional integrators or preserves certain properties of the underlying system.

\section*{Data availability}

The QENDy code and examples that support the findings presented in this paper are available at \url{https://github.com/sklus/d3s/}.

\section*{Acknowledgments}

We would like to thank Feliks Nüske and Peter Benner for interesting discussions about the quadratic embedding framework. J.-P.\ was supported by the EPSRC Centre for Doctoral Training in Mathematical Modelling, Analysis and Computation (MAC-MIGS) funded by the UK Engineering and Physical Sciences Research Council (grant EP/S023291/1), Heriot--Watt University and the University of Edinburgh.

\bibliographystyle{unsrturl}
\bibliography{QENDy.bib}

\begin{thebibliography}{10}

\bibitem{Ko31}
B.~Koopman.
\newblock Hamiltonian systems and transformations in {H}ilbert space.
\newblock {\em Proceedings of the National Academy of Sciences}, 17(5):315,
  1931.
\newblock \href {https://doi.org/10.1073/pnas.17.5.315}
  {\path{doi:10.1073/pnas.17.5.315}}.

\bibitem{LaMa94}
A.~Lasota and M.~C. Mackey.
\newblock {\em Chaos, fractals, and noise: Stochastic aspects of dynamics},
  volume~97 of {\em Applied Mathematical Sciences}.
\newblock Springer, New York, 2nd edition, 1994.

\bibitem{Mezic05}
I.~Mezi{\'{c}}.
\newblock Spectral properties of dynamical systems, model reduction and
  decompositions.
\newblock {\em Nonlinear Dynamics}, 41(1):309--325, 2005.
\newblock \href {https://doi.org/10.1007/s11071-005-2824-x}
  {\path{doi:10.1007/s11071-005-2824-x}}.

\bibitem{BMM12}
M.~Budi{\v s}i{\'c}, R.~Mohr, and I.~Mezi{\'c}.
\newblock Applied {K}oopmanism.
\newblock {\em Chaos: An Interdisciplinary Journal of Nonlinear Science},
  22(4), 2012.
\newblock \href {https://doi.org/10.1063/1.4772195}
  {\path{doi:10.1063/1.4772195}}.

\bibitem{WKR15}
M.~O. Williams, I.~G. Kevrekidis, and C.~W. Rowley.
\newblock A data-driven approximation of the {K}oopman operator: Extending
  dynamic mode decomposition.
\newblock {\em Journal of Nonlinear Science}, 25(6):1307--1346, 2015.
\newblock \href {https://doi.org/10.1007/s00332-015-9258-5}
  {\path{doi:10.1007/s00332-015-9258-5}}.

\bibitem{KKS16}
S.~Klus, P.~Koltai, and C.~Sch{\"u}tte.
\newblock On the numerical approximation of the {P}erron--{F}robenius and
  {K}oopman operator.
\newblock {\em Journal of Computational Dynamics}, 3(1):51--79, 2016.
\newblock \href {https://doi.org/10.3934/jcd.2016003}
  {\path{doi:10.3934/jcd.2016003}}.

\bibitem{SS13}
C.~Sch\"utte and M.~Sarich.
\newblock {\em Metastability and Markov State Models in Molecular Dynamics:
  Modeling, Analysis, Algorithmic Approaches}.
\newblock Number~24 in Courant Lecture Notes. American Mathematical Society,
  2013.

\bibitem{MauMez16}
A.~Mauroy and I.~Mezi{\'{c}}.
\newblock Global stability analysis using the eigenfunctions of the {K}oopman
  operator.
\newblock {\em IEEE Transactions on Automatic Control}, 61(11):3356--3369,
  2016.
\newblock \href {https://doi.org/10.1109/TAC.2016.2518918}
  {\path{doi:10.1109/TAC.2016.2518918}}.

\bibitem{MauGon16}
A.~Mauroy and J.~Goncalves.
\newblock Linear identification of nonlinear systems: {A} lifting technique
  based on the {K}oopman operator.
\newblock In {\em 2016 IEEE 55th Conference on Decision and Control (CDC)},
  pages 6500--6505, 2016.
\newblock \href {https://doi.org/10.1109/CDC.2016.7799269}
  {\path{doi:10.1109/CDC.2016.7799269}}.

\bibitem{KorMez16}
M.~Korda and I.~Mezi{\'{c}}.
\newblock Linear predictors for nonlinear dynamical systems: {K}oopman operator
  meets model predictive control.
\newblock {\em Automatica}, 93:149--160, 2018.
\newblock \href {https://doi.org/10.1016/j.automatica.2018.03.046}
  {\path{doi:10.1016/j.automatica.2018.03.046}}.

\bibitem{PeiKlul19}
S.~Peitz and S.~Klus.
\newblock Koopman operator-based model reduction for switched-system control of
  {PDE}s.
\newblock {\em Automatica}, 106:184--191, 2019.
\newblock \href {https://doi.org/10.1016/j.automatica.2019.05.016}
  {\path{doi:10.1016/j.automatica.2019.05.016}}.

\bibitem{KNPNCS20}
S.~Klus, F.~N\"uske, S.~Peitz, J.-H. Niemann, C.~Clementi, and C.~Sch\"utte.
\newblock Data-driven approximation of the {K}oopman generator: {M}odel
  reduction, system identification, and control.
\newblock {\em Physica D: Nonlinear Phenomena}, 406:132416, 2020.
\newblock \href {https://doi.org/10.1016/j.physd.2020.132416}
  {\path{doi:10.1016/j.physd.2020.132416}}.

\bibitem{BSH21}
P.~Bevanda, S.~Sosnowski, and S.~Hirche.
\newblock Koopman operator dynamical models: Learning, analysis and control.
\newblock {\em Annual Reviews in Control}, 52:197--212, 2021.
\newblock \href {https://doi.org/doi.org/10.1016/j.arcontrol.2021.09.002}
  {\path{doi:doi.org/10.1016/j.arcontrol.2021.09.002}}.

\bibitem{KD24}
S.~Klus and N.~Djurdjevac Conrad.
\newblock Dynamical systems and complex networks: A {K}oopman operator
  perspective.
\newblock {\em Journal of Physics: Complexity}, 2024.
\newblock \href {https://doi.org/10.1088/2632-072X/ad9e60}
  {\path{doi:10.1088/2632-072X/ad9e60}}.

\bibitem{Car32}
T.~Carleman.
\newblock Application de la théorie des équations intégrales linéaires aux
  systèmes d'équations différentielles non linéaires.
\newblock {\em Acta Mathematica}, 59:63--87, 1932.
\newblock \href {https://doi.org/10.1007/BF02546499}
  {\path{doi:10.1007/BF02546499}}.

\bibitem{KS91}
K.~Kowalski and W.-H. Steeb.
\newblock {\em Nonlinear dynamical systems and Carleman linearization}.
\newblock World Scientific, 1991.

\bibitem{DY24}
D.~Shi and X.~Yang.
\newblock Koopman spectral linearization vs. {C}arleman linearization: A
  computational comparison study.
\newblock {\em Mathematics}, 12(14), 2024.
\newblock \href {https://doi.org/10.3390/math12142156}
  {\path{doi:10.3390/math12142156}}.

\bibitem{Ker81}
E.~H. Kerner.
\newblock Universal formats for nonlinear ordinary differential systems.
\newblock {\em Journal of Mathematical Physics}, 22(7):1366--1371, 1981.

\bibitem{Gu11}
C.~Gu.
\newblock {QLMOR}: A projection-based nonlinear model order reduction approach
  using quadratic-linear representation of nonlinear systems.
\newblock {\em IEEE Transactions on Computer-Aided Design of Integrated
  Circuits and Systems}, 30(9):1307--1320, 2011.

\bibitem{App02}
G.~G. Appelroth.
\newblock La forme fondamentale du système d'équations différentielles
  algébriques.
\newblock {\em Matematicheskii Sbornik}, 23(1):12--23, 1902.

\bibitem{SV87}
M.~A. Savageau and E.~O. Voit.
\newblock Recasting nonlinear differential equations as {S}-systems: a
  canonical nonlinear form.
\newblock {\em Mathematical Biosciences}, 87(1):83--115, 1987.

\bibitem{PP05}
A.~Papachristodoulou and S.~Prajna.
\newblock Analysis of non-polynomial systems using the sum of squares
  decomposition.
\newblock In {\em Positive polynomials in control}, pages 23--43. Springer
  Berlin Heidelberg, 2005.

\bibitem{LZZZ15}
J.~Liu, N.~Zhan, H.~Zhao, and L.~Zou.
\newblock Abstraction of elementary hybrid systems by variable transformation.
\newblock In {\em International Symposium on Formal Methods}, pages 360--377,
  Cham, 2015. Springer International Publishing.

\bibitem{Car15}
F.~Carravetta.
\newblock Global exact quadratization of continuous-time nonlinear control
  systems.
\newblock {\em SIAM Journal on Control and Optimization}, 53(1):235--261, 2015.
\newblock \href {https://doi.org/10.1137/130915418}
  {\path{doi:10.1137/130915418}}.

\bibitem{KW19}
B.~Kramer and K.~E. Willcox.
\newblock Nonlinear model order reduction via lifting transformations and
  proper orthogonal decomposition.
\newblock {\em AIAA Journal}, 57(6):2297--2307, 2019.

\bibitem{CP24}
Y.~Cai and G.~Pogudin.
\newblock Dissipative quadratizations of polynomial {ODE} systems.
\newblock In {\em International Conference on Tools and Algorithms for the
  Construction and Analysis of Systems}, pages 323--342, 2024.

\bibitem{BIPK24}
A.~Bychkov, O.~Issan, G.~Pogudin, and B.~Kramer.
\newblock Exact and optimal quadratization of nonlinear finite-dimensional
  nonautonomous dynamical systems.
\newblock {\em SIAM Journal on Applied Dynamical Systems}, 23(1):982--1016,
  2024.
\newblock \href {https://doi.org/10.1137/23M1561129}
  {\path{doi:10.1137/23M1561129}}.

\bibitem{GB24}
P.~Goyal and P.~Benner.
\newblock Generalized quadratic embeddings for nonlinear dynamics using deep
  learning.
\newblock {\em Physica D: Nonlinear Phenomena}, 463:134158, 2024.
\newblock \href {https://doi.org/10.1016/j.physd.2024.134158}
  {\path{doi:10.1016/j.physd.2024.134158}}.

\bibitem{BPK16}
S.~L. Brunton, J.~L. Proctor, and J.~N. Kutz.
\newblock Discovering governing equations from data by sparse identification of
  nonlinear dynamical systems.
\newblock {\em Proceedings of the National Academy of Sciences},
  113(15):3932--3937, 2016.

\bibitem{PW16}
B.~Peherstorfer and K.~Willcox.
\newblock Data-driven operator inference for nonintrusive projection-based
  model reduction.
\newblock {\em Computer Methods in Applied Mechanics and Engineering},
  306:196--215, 2016.
\newblock \href {https://doi.org/10.1016/j.cma.2016.03.025}
  {\path{doi:10.1016/j.cma.2016.03.025}}.

\bibitem{NoNu13}
F.~No{\'e} and F.~N{\"u}ske.
\newblock A variational approach to modeling slow processes in stochastic
  dynamical systems.
\newblock {\em Multiscale Modeling \& Simulation}, 11(2):635--655, 2013.

\bibitem{Ban06}
M.~C. Ba\~nuls, R.~Or\'us, J.~I. Latorre, A.~P\'erez, and P.~Ruiz-Femenia.
\newblock Simulation of many-qubit quantum computation with matrix product
  states.
\newblock {\em Physical Review A}, 73:022344, 2006.
\newblock \href {https://doi.org/10.1103/PhysRevA.73.022344}
  {\path{doi:10.1103/PhysRevA.73.022344}}.

\bibitem{Ose11}
I.~Oseledets.
\newblock Tensor-train decomposition.
\newblock {\em SIAM Journal on Scientific Computing}, 33:2295--2317, 2011.
\newblock \href {https://doi.org/10.1137/090752286}
  {\path{doi:10.1137/090752286}}.

\bibitem{PP12}
K.~B. Petersen and M.~S. Pedersen.
\newblock The {M}atrix {C}ookbook, 2012.

\bibitem{Kra21}
B.~Kramer.
\newblock Stability domains for quadratic-bilinear reduced-order models.
\newblock {\em SIAM Journal on Applied Dynamical Systems}, 20(2):981--996,
  2021.
\newblock \href {https://doi.org/10.1137/20M1364849}
  {\path{doi:10.1137/20M1364849}}.

\bibitem{SKP23}
N.~Sawant, B.~Kramer, and B.~Peherstorfer.
\newblock Physics-informed regularization and structure preservation for
  learning stable reduced models from data with operator inference.
\newblock {\em Computer Methods in Applied Mechanics and Engineering},
  404:115836, 2023.
\newblock \href {https://doi.org/10.1016/j.cma.2022.115836}
  {\path{doi:10.1016/j.cma.2022.115836}}.

\bibitem{DP98}
G-R. Duan and R.J. Patton.
\newblock A note on {H}urwitz stability of matrices.
\newblock {\em Automatica}, 34(4):509--511, 1998.
\newblock \href {https://doi.org/10.1016/S0005-1098(97)00217-3}
  {\path{doi:10.1016/S0005-1098(97)00217-3}}.

\bibitem{GPB23}
P.~Goyal, I.~Pontes Duff, and P.~Benner.
\newblock Guaranteed stable quadratic models and their applications in sindy
  and operator inference, 2023.
\newblock URL: \url{https://arxiv.org/abs/2308.13819}, \href
  {http://arxiv.org/abs/2308.13819} {\path{arXiv:2308.13819}}.

\bibitem{BNC18}
L.~Boninsegna, F.~Nüske, and C.~Clementi.
\newblock Sparse learning of stochastic dynamical equations.
\newblock {\em The Journal of Chemical Physics}, 148(24):241723, 2018.
\newblock \href {https://doi.org/10.1063/1.5018409}
  {\path{doi:10.1063/1.5018409}}.

\bibitem{RBPK17}
S.~H. Rudy, S.~L. Brunton, J.~L. Proctor, and J.~N. Kutz.
\newblock Data-driven discovery of partial differential equations.
\newblock {\em Science Advances}, 3(4):e1602614, 2017.
\newblock \href {https://doi.org/10.1126/sciadv.1602614}
  {\path{doi:10.1126/sciadv.1602614}}.

\bibitem{Thomas99}
R.~Thomas.
\newblock Deterministic chaos seen in terms of feedback circuits: Analysis,
  synthesis, ``labyrinth chaos''.
\newblock {\em International Journal of Bifurcation and Chaos},
  9(10):1889--1905, 1999.
\newblock \href {https://doi.org/10.1142/S0218127499001383}
  {\path{doi:10.1142/S0218127499001383}}.

\bibitem{EVNE20}
H.~Eivazi, H.~Veisi, M.~H. Naderi, and V.~Esfahanian.
\newblock Deep neural networks for nonlinear model order reduction of unsteady
  flows.
\newblock {\em Physics of Fluids}, 32(10):105104, 2020.
\newblock \href {https://doi.org/10.1063/5.0020526}
  {\path{doi:10.1063/5.0020526}}.

\bibitem{GKES19}
P.~Gel{\ss}, S.~Klus, J.~Eisert, and C.~Sch{\"u}tte.
\newblock Multidimensional approximation of nonlinear dynamical systems.
\newblock {\em Journal of Computational and Nonlinear Dynamics}, 14:061006,
  2019.
\newblock \href {https://doi.org/10.1115/1.4043148}
  {\path{doi:10.1115/1.4043148}}.

\bibitem{KL97}
W.~Kahan and R.-C. Li.
\newblock Unconventional schemes for a class of ordinary differential
  equations---with applications to the {K}orteweg--de {V}ries equation.
\newblock {\em Journal of Computational Physics}, 134(2):316--331, 1997.

\bibitem{Deu01}
F.~Deutsch.
\newblock {\em Best approximation in inner product spaces}.
\newblock CMS Books in Mathematics. Springer, 2001.

\bibitem{SK11}
H.~R. Schwarz and N.~Köckler.
\newblock {\em Numerische Mathematik}.
\newblock Vieweg+Teubner Verlag, Wiesbaden, 8th edition, 2011.
\newblock \href {https://doi.org/10.1007/978-3-8348-8166-3}
  {\path{doi:10.1007/978-3-8348-8166-3}}.

\bibitem{GVL13}
G.~H. Golub and C.~F.~Van Loan.
\newblock {\em Matrix Computations}.
\newblock Johns Hopkins University Press, 4th edition, 2013.

\end{thebibliography}

\appendix

\section{Proof of Theorem \ref{thm:derivatives}}
\label{app:Proof}

We first rewrite the loss function as
\begin{align*}
    L(A, B, C) &= \norm{\dot{Z} - A \ts Z_2 - B \ts Z_1 - C \ts \mathds{1}_m^\top}_F^2 \\
    &= \underbrace{\tr\left(\dot{Z}^\top \dot{Z}\right)}_{\circled{0}}
     + \underbrace{\tr\left(Z_2^\top A^\top A \ts Z_2\right)}_{\circled{1}}
     + \underbrace{\tr\left(Z_1^\top B^\top B \ts Z_1\right)}_{\circled{2}}
     + \underbrace{\tr\left(\mathds{1}_m C^\top C \ts \mathds{1}_m^\top\right)}_{\circled{3}} \\
    &- 2 \underbrace{\tr\left(\dot{Z}^\top A \ts Z_2\right)}_{\circled{4}}
     - 2 \underbrace{\tr\left(\dot{Z}^\top B \ts Z_1\right)}_{\circled{5}}
     - 2 \underbrace{\tr\left(\dot{Z}^\top C \ts \mathds{1}_m^\top\right)}_{\circled{6}} \label{eq:loss function} \\
    &+ 2 \underbrace{\tr\left(Z_2^\top A^\top B \ts Z_1\right)}_{\circled{7}}
     + 2 \underbrace{\tr\left(Z_2^\top A^\top C \ts \mathds{1}_m^\top\right)}_{\circled{8}}
     + 2 \underbrace{\tr\left(Z_1^\top B^\top C \ts \mathds{1}_m^\top\right)}_{\circled{9}}.
\end{align*}
Let us now compute the derivatives with respect to the matrix $ A $. Only the terms \circled{1}, \circled{4}, \circled{7}, and \circled{8} depend on $ A $. It holds that
\begin{equation*}
    \pd{}{A} \tr\big(X \ts A \ts Y\big) = X^\top Y^\top
    \quad \text{and} \quad
    \pd{}{A} \tr\big(A^\top X \ts A \ts Y\big) = X A \ts Y + X^\top \! A \ts Y^\top,
\end{equation*}
see, e.g., \cite{PP12}. For term \circled{1}, using the cyclic property of the trace, this implies
\begin{equation*}
     \pd{}{A} \tr\left(Z_2^\top A^\top A \ts Z_2\right) = \pd{}{A} \tr\left(A^\top A \ts Z_2 \ts Z_2^\top\right) = 2 \ts A \ts Z_2 \ts Z_2^\top.
\end{equation*}
Similarly, for the terms \circled{4}, \circled{7}, and \circled{8}, we obtain
\begin{align*}
    \pd{}{A} \tr\left(\dot{Z}^\top A \ts Z_2\right) &= \dot{Z} \ts Z_2^\top, \\
    \pd{}{A} \tr\left(Z_2^\top A^\top B \ts Z_1\right) &= B \ts Z_1 \ts Z_2^\top, \\
    \pd{}{A} \tr\left(Z_2^\top A^\top C \ts \mathds{1}_m^\top\right) &= C \ts \mathds{1}_m^\top Z_2^\top,
\end{align*}
where we used $ \tr\big(X\big) = \tr\big(X^\top\big) $. The derivatives with respect to $ B $ follow in an analogous fashion. Finally, we compute the derivatives with respect to $ C $. For term \circled{3}, we get
\begin{equation*}
    \pd{}{C} \tr\left(\mathds{1}_m \ts C^\top C \ts \mathds{1}_m^\top\right) = \pd{}{C} \tr\left(\mathds{1}_m^\top \ts \mathds{1}_m \ts C^\top C\right) = m \pd{}{C} \tr\left(C^\top C\right) = 2 \ts m \ts C.
\end{equation*}
The remaining derivatives for the terms \circled{6}, \circled{8}, and \circled{9} can be computed in the same way.

\section{Best approximation}
\label{app:Best approximation}

Let $ H $ be a real Hilbert space with inner product $ \innerprod{\cdot}{\cdot} $ and induced norm $ \norm{\cdot} = \innerprod{\cdot}{\cdot}^{\frac{1}{2}} $. Furthermore, let $ \{ \phi_1, \dots, \phi_n \} \subset{H} $ be a set of basis vectors spanning the linear subspace $ V = \mspan\{\phi_1, \dots, \phi_n \} $. Any element $ g \in V $ can be written as
\begin{equation*}
    g = \sum_{i=1}^n \alpha_i \ts \phi_i = \alpha^\top \phi,
\end{equation*}
where $ \alpha = [\alpha_1, \dots, \alpha_n]^\top \in \R^n $. Note that this representation is unique only if the basis vectors are linearly independent. Given an arbitrary vector $ f \in H $, we want to find $ g \in V $ such that the error
\begin{equation} \label{eq:best approximation problem}
    E(\alpha) := \norm{f - g}^2 = \innerprod{f}{f} - 2 \ts \sum_{i=1}^n \alpha_i \ts \innerprod{\phi_i}{f} + \sum_{i=1}^n \sum_{j=1}^n \alpha_i \ts \alpha_j \ts \innerprod{\phi_i}{\phi_j}
\end{equation}
is minimized, i.e., $ g $ is the best approximation of $ f $ contained in the subspace $ V $. Defining $ A \in \R^{n \times n} $ and $ b \in \R^n $, with $ a_{ij} = \innerprod{\phi_i}{\phi_j} $ and $ b_i = \innerprod{\phi_i}{f} $, we have
\begin{equation*}
    E(\alpha) = \innerprod{f}{f} - 2 \ts b^\top \alpha + \alpha^\top A \ts \alpha
    \implies
    \nabla E(\alpha) = -2 \ts b + 2 \ts A \ts \alpha.
\end{equation*}
In order to find optimal solutions, we thus have to solve the system of linear equations
\begin{equation} \label{eq:best approximation}
    A \ts \alpha = b.
\end{equation}

\begin{theorem}
Let the basis vectors $ \{ \phi_1, \dots, \phi_n \} $ be linearly independent. Then the matrix $ A $ is symmetric and positive definite and the unique minimizer of \eqref{eq:best approximation problem} is given by \eqref{eq:best approximation}.
\end{theorem}

\begin{proof}
See, e.g., \cite{Deu01, SK11}.
\end{proof}

The matrix $ A $ is non-singular if and only if the basis vectors are linearly independent. In our setting, however, this is in general not the case. Assume that
\begin{equation*}
    V = \mspan\{\phi_1, \dots, \phi_n \} = \mspan\{\phi_1, \dots, \phi_r \},
\end{equation*}
with $ r < n $, i.e., the basis vectors span only an $ r $-dimensional vector space. Thus,
\begin{equation*}
    \phi_k = \sum_{i=1}^r \beta_{ki} \ts \phi_i, \quad \text{for } k = r+1, \dots, n.
\end{equation*}
Consider now the $ k $th rows of $ A $ and $ b $, with $ k > r $, then
\begin{align*}
    A_{kj} &= \innerprod{\phi_k}{\phi_j} = \sum_{i=1}^r \beta_{ki} \innerprod{\phi_i}{\phi_j} = \sum_{i=1}^r \beta_{ki} \ts a_{ij}, \\
    b_k &= \innerprod{\phi_k}{f} = \sum_{i=1}^r \beta_{ki} \ts \innerprod{\phi_i}{f} = \sum_{i=1}^r \beta_{ki} \ts b_i.
\end{align*}
The last $ n-r $ rows of $ A $ and $ b $ are hence the same linear combinations of the first $ r $ rows of $ A $ and~$ b $. It follows that the rank of the augmented matrix $ [A \mid b] $ is $ r $, which is also the rank of $ A $. That is, the system \eqref{eq:best approximation} is consistent and there is an infinite number of solutions. Let $ \widetilde{A} $ be the submatrix comprising the first $ r $ rows and columns of $ A $, $ \widetilde{b} $ the first $ r $ rows of $ b $, and $ \widetilde{\alpha} $ the unique solution of $ \widetilde{A} \ts \widetilde{\alpha} = \widetilde{b} $, then the vector
\begin{equation*}
     \alpha_p =
     \begin{bmatrix}
        \widetilde{\alpha} \\
        0
     \end{bmatrix} \in \R^n
\end{equation*}
is a particular solution of \eqref{eq:best approximation}. Adding any vector $ u $ contained in the null space of $ A $ will also result in a valid solution. Defining $ \alpha = A^+ b $ picks the solution with the smallest 2-norm, see \cite{GVL13}.

\begin{example}
\begin{figure}
    \centering
    \begin{minipage}{0.45\linewidth}
        \centering
        \subfiguretitle{(a)}
        \vspace*{0.5ex}
        \includegraphics[width=0.95\linewidth]{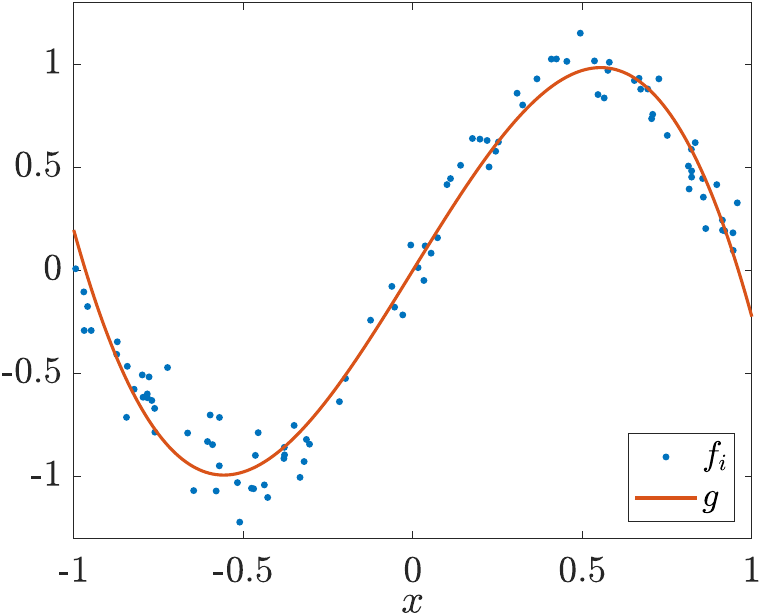}
    \end{minipage}
    \begin{minipage}{0.45\linewidth}
        \centering
        \subfiguretitle{(b)}
        \vspace*{0.5ex}
        \includegraphics[width=0.95\linewidth]{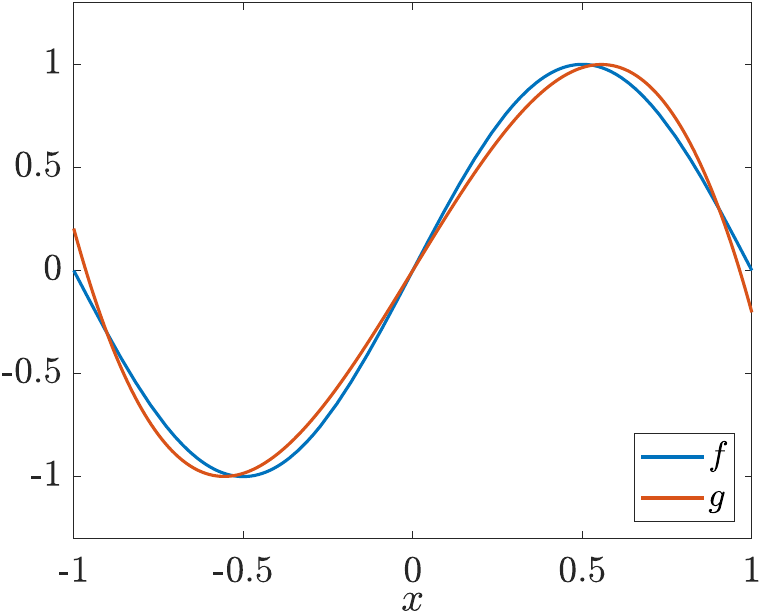}
    \end{minipage}
    \caption{(a) Approximation of the data points $ (x_i, f_i) $ by a polynomial of degree 3. (b) Approximation of $ \sin(\pi \ts x) $ by a polynomial of degree 3.}
    \label{fig:best approximation}
\end{figure}

We have to distinguish between discrete and continuous best approximation:
\begin{enumerate}[leftmargin=3.5ex, itemsep=0ex, topsep=0.5ex, label=\roman*)]
\item Assume we have data $ \{(x_i, f_i) \}_{i=1}^m $, with $ x_i, f_i \in \R $ and we want to find a function
\begin{equation*}
    g(x) = \sum_{i=1}^n \alpha_i \ts \widehat{\phi}_i(x)
\end{equation*}
that approximates the data. We define the vectors $ \phi_i = [\widehat{\phi}_i(x_1), \dots, \widehat{\phi}_i(x_m)]^\top \in \R^m $ and $ f = [f_1, \dots, f_m]^\top \in \R^m $. Solving the resulting system of linear equations \eqref{eq:best approximation}, we obtain the coefficients $ \alpha $. The Hilbert space is $ H = \R^m $, with inner product $ \innerprod{f}{g} = \sum_{i=1}^m f_i \ts g_i \ts \mu_i $, where $ \mu_i > 0 $ is the weight for the data point $ x_i $.
\item Assume we want to approximate the function $ f(x) $ in the interval $ [a, b] $ by a function of the form
\begin{equation*}
    g(x) = \sum_{i=1}^n \alpha_i \ts \phi_i(x).
\end{equation*}
We choose $ H = L_\mu^2([a, b]) $, i.e., the space of square-integrable functions with respect to a measure~$ \mu $, with inner product $ \innerprod{f}{g}_\mu = \int_a^b f(x) \ts g(x) \ts \mathrm{d}\mu(x) $, and again compute $ \alpha $ by solving \eqref{eq:best approximation}.
\end{enumerate}
The difference between discrete and continuous best approximation problems is also illustrated in Figure~\ref{fig:best approximation}. \exampleSymbol
\end{example}

\section{Quadratic embedding of the modified Thomas attractor}
\label{app:Modified Thomas attractor}

As shown in Section~\ref{sec:Numerical results}, the nine-dimensional feature space defined in \eqref{eq:Thomas embedding} is not sufficient for the modified Thomas attractor. In order to find a quadratic embedding, we augment the feature space as follows:
\begin{equation*}
    \renewcommand*{\arraystretch}{1.05}
    z =
    \begin{bmatrix}
        x_1 \\
        x_2 \\
        x_3 \\
        \sin(x_1) \\
        \sin(x_2) \\
        \sin(x_3) \\
        \cos(x_1) \\
        \cos(x_2) \\
        \cos(x_3) \\
        x_2 \ts \sin(x_1) \\
        x_3 \ts \sin(x_2) \\
        x_1 \ts \sin(x_3) \\
        x_2 \ts \cos(x_1) \\
        x_3 \ts \cos(x_2) \\
        x_1 \ts \cos(x_3)
    \end{bmatrix}
    \implies
    \dot{z} =
    \begin{bmatrix}
        z_5 - \alpha \ts z_1 - \beta \ts z_{13} \\
        z_6 - \alpha \ts z_2 - \beta \ts z_{14} \\
        z_4 - \alpha \ts z_3 - \beta \ts z_{15} \\
        z_5 \ts z_7 - \alpha \ts z_1 \ts z_7 - \beta \ts z_7 \ts z_{13} \\
        z_6 \ts z_8 - \alpha \ts z_2 \ts z_8 - \beta \ts z_8 \ts z_{14} \\
        z_4 \ts z_9 - \alpha \ts z_3 \ts z_9 - \beta \ts z_9 \ts z_{15} \\
        -z_4 \ts z_5 + \alpha \ts z_1 \ts z_4 + \beta \ts z_4 \ts z_{13} \\
        -z_5 \ts z_6 + \alpha \ts z_2 \ts z_5 + \beta \ts z_5 \ts z_{14} \\
        -z_4 \ts z_6 + \alpha \ts z_3 \ts z_6 + \beta \ts z_1 \ts z_{15} \\
        (z_4 \ts z_6 + z_5 \ts z_{13}) - \alpha \ts (z_1 \ts z_{13} + z_2 \ts z_4) - \beta \ts (z_4 \ts z_{14} + z_{13}^2) \\
        (z_4 \ts z_5 + z_6 \ts z_{14}) - \alpha \ts (z_2 \ts z_{14} + z_3 \ts z_5) - \beta \ts (z_5 \ts z_{15} + z_{14}^2) \\
        (z_4 \ts z_{15} + z_5 \ts z_6) - \alpha \ts (z_1 \ts z_6 + z_3 \ts z_{15}) - \beta \ts (z_6 \ts z_{13} + z_{15}^2) \\
        (-z_5 \ts z_{10} + z_6 \ts z_7) - \alpha (-z_1 \ts z_{10} + \alpha \ts z_2 \ts z_7) - \beta \ts (z_7 \ts z_{14} - \ts z_{10} \ts z_{13}) \\
        (z_4 \ts z_8 - z_6 \ts z_{11}) - \alpha \ts (-z_2 \ts z_{11} + z_3 \ts z_8) - \beta \ts (-z_{11} \ts z_{14} + z_8 \ts z_{15}) \\
       (-z_4 \ts z_{12} + z_5 \ts z_9) - \alpha \ts (z_1 \ts z_9 - \alpha \ts z_3  \ts z_{12}) - \beta \ts (z_9 \ts z_{13} - \beta \ts z_{12} \ts z_{15})
    \end{bmatrix}.
\end{equation*}
The representation, however, is not unique. Using this set of basis functions, QENDy successfully learns a valid quadratic model and the matrices $ A $ and $ B $ are again sparse.

\end{document}